\documentclass{amsart}
\pdfoutput=1

\usepackage{amssymb}
\usepackage{microtype}
\usepackage{tikz}
\usetikzlibrary{arrows}
\usepackage{booktabs}
\usepackage[pdftex,colorlinks,citecolor=black,linkcolor=black,urlcolor=black,bookmarks=false]{hyperref}
\def\MR#1{\href{http://www.ams.org/mathscinet-getitem?mr=#1}{MR#1}}
\def\arXiv#1{arXiv:\href{http://arXiv.org/abs/#1}{#1}}
\usepackage{doi}
\usepackage{subfigure}
\usepackage{bbm} 
\usepackage{xcolor}
\usepackage{colortbl}

\newtheorem{theorem}{Theorem}[section]
\newtheorem{proposition}[theorem]{Proposition}

\newtheorem{corollary}[theorem]{Corollary}

\theoremstyle{definition}

\numberwithin{figure}{section}
\numberwithin{equation}{section}
\numberwithin{table}{section}

\newcommand{\Z}{\mathbb{Z}}
\newcommand{\Q}{\mathbb{Q}}
\newcommand{\R}{\mathbb{R}}
\newcommand{\C}{\mathbb{C}}

\newcommand{\SL}{{\rm SL}}
\newcommand{\lcm}{\operatorname{lcm}}

\renewcommand{\Im}{\operatorname{Im}}
\renewcommand{\Re}{\operatorname{Re}}

\title{Six-dimensional sphere packing and linear programming}

\author{Matthew de Courcy-Ireland, Maria Dostert, Maryna Viazovska}
\address{Institute of Mathematics\\
EPFL\\
CH-1015 Lausanne, Switzerland and Department of Mathematics, Stockholm University} \email{matthew.decourcy-ireland@math.su.se}
\address{Department of Mathematics\\
  KTH Royal Institute of Technology, Stockholm, Sweden}
\email{maria.dostert@gmail.com}
\address{Institute of Mathematics\\
EPFL\\
CH-1015 Lausanne, Switzerland} \email{viazovska@gmail.com}

\date{July 26, 2023}

\begin{document}

\begin{abstract}
We prove that the Cohn--Elkies linear programming bound for sphere packing is not sharp in dimension 6. The proof uses duality and optimization over a space of modular forms, generalizing a construction of Cohn--Triantafillou to the case of odd weight and non-trivial character.
\end{abstract}

\maketitle

\section{Introduction} \label{sec:introduction}

The method of linear programming gives an upper bound for the density achievable by sphere packings in Euclidean space of any given dimension. In some cases, especially dimensions 8 and 24 as shown in \cite{V,CKMRV}, the bound is equal to the density of a known configuration, giving a proof of optimality. In most dimensions, we do not know the optimal packing, and we have only a numerical approximation to the linear programming bound. The optimal density and the linear programming bound are not expected to be equal in general.
It seems plausible that equality occurs only in dimensions 1, 2, 8, and 24, which are the values conjectured by Cohn and Elkies \cite[Conjecture 7.3]{CE}, with 2 the last of the expected sharp cases to defy proof.
However, it has only recently become possible to rule out the scenario that linear programming gives a sharp bound in all dimensions: it provably fails in dimensions 3, 4, 5, 12, and 16 by work of \cite[Corollary 1.3]{CdLS}, \cite[sections 6-7]{L} for the cases 3, 4, 5, and \cite[Table 6.1]{CT} for 12 and 16.
In this article, we prove that the bound is not sharp in dimension 6. Intuitively, the reason is that the linear programming method applies not only to packings, but also to ``fake packings" with impossibly high densities. We construct one such example.

The first obstacle is that the densest packing in dimension 6 is not known. We therefore compare the linear programming bound (LP bound) to a more powerful bound, based on semidefinite programming (SDP) and developed by Cohn, de Laat, and Salmon. We refer to \cite[Table 1.1]{CdLS} for bounds in a range of dimensions.
\begin{theorem}[Cohn, de Laat, Salmon] \label{thm:sdp}
The sphere packing density in dimension $6$ is at most $0.410304$. 
\end{theorem}
In contrast, we prove that the LP bound is constrained by the following ``bound on the bound".
\begin{theorem} \label{thm:main}
The LP bound for density in dimension $6$ is at least $0.410948$.
\end{theorem}
\begin{corollary} \label{cor:cor}
The linear programming bound in dimension $6$ is strictly higher than the density of any packing.
\end{corollary}

For comparison, the highest density known in dimension 6 is
\begin{equation} \label{eqn:e6-density}
\frac{\pi^3}{48\sqrt{3}} = 0.372947 \ldots
\end{equation}
This is achieved by the $E_6$ root lattice \cite[p. 126]{splag}, which is closely related to the sharp case of $E_8$ in dimension 8. Numerical optimization over a restricted class of functions in the linear program gives an upper bound 0.417674, which is presumably very close to the true value, but in principle might differ considerably from the optimum over the full infinite-dimensional space of candidates.
Theorem~\ref{thm:main} shows that the true value cannot be much lower, and in particular it exceeds both the highest known density (\ref{eqn:e6-density}) and the upper bound from Theorem~\ref{thm:sdp}. The proof sheds some light on why the linear programming bound is not sharp in dimension 6, as it is for dimension 8, despite the close connection between the exceptional structures $E_6$ and $E_8$.
On the other hand, it is known that there are infinitely many other packings achieving the same density as $E_6$, but with no particularly close connection to $E_8$. They are obtained by stacking layers of the (presumed) densest 5-dimensional packing \cite[p. 144]{splag}.

The proof of Theorem~\ref{thm:main} uses duality for linear programs. To describe it further, let us recall more specifically how the LP bound works.
It is convenient to work with the ``center density" of a packing, which is the density divided by the volume of a ball in $\R^D$ of the corresponding radius. This simplifies various factors of $\pi$, and can be thought of as the number of centers per unit volume.
For example, the center density of $E_6$ is the value from (\ref{eqn:e6-density}) divided by $\pi^{D/2}/(D/2)! = \pi^3/6$, namely
\begin{equation} \label{eqn:e6-center-density}
\frac{\pi^3}{48\sqrt{3}} \div \frac{\pi^3}{6} = \frac{1}{8\sqrt{3}} = 0.072168\ldots
\end{equation}
Theorem~\ref{thm:sdp} implies that all packings have center density at most 0.079398.
Theorem~\ref{thm:main} shows that the LP bound for center density is at least 0.0795223, and numerical optimization leads to 0.08084 as an upper bound.
In fact, our number 0.0795223 is rounded down from a quadratic irrational. What we really show is that the LP bound for center density in dimension 6 is at least
\begin{align}
 &\frac{- 277385984684414834701547634832199852580621960702176236773103}{535700179589322461444902359627590796379300404334023027566592} \label{exact-quadratic} \\
&+ \frac{554232205790268185636220216828951751933789602521848882511869}{1607100538767967384334707078882772389137901213002069082699776}\sqrt{3} \nonumber 
\end{align}
hence greater than
\begin{equation} \label{eqn:digits}
0.079522333845052286373845030218205166528
\end{equation}
The fractions in (\ref{exact-quadratic}) work out to $\alpha + \beta\sqrt{3}$ where $\alpha \approx -0.5178$ and $\beta \approx 0.3448$.

\begin{theorem}[Cohn--Elkies, \cite{CE}] \label{thm:cohn-elkies}
If an integrable, continuous function $f: \R^D \rightarrow \R$ with Fourier transform $\widehat{f}$ is not identically $0$ and satisfies $f(x) \leq 0$ for $|x| \geq r$, and $\widehat{f} \geq 0$,
then the center density of any sphere packing in $\R^{D}$ is at most
\begin{equation} \label{eqn:ce}
\left( \frac{r}{2} \right)^D f(0) \div \widehat{f}(0)
\end{equation}
\end{theorem}

By the Cohn--Elkies linear programming bound, or simply the LP bound, we mean the infimum of (\ref{eqn:ce}) over all admissible $f$. We have assumed $f$ is continuous and integrable so that $f(0)$ and $\widehat{f}(0) = \int f$ are defined in the most straightforward way. It is also of interest to take the infimum over other classes of functions (for instance, Schwartz), and these variants could also be called ``the" LP bound. 
In fact, the original formulation from \cite{CE} assumed that $f$ is sufficiently smooth and rapidly decaying so that both sides of the Poisson summation formula converge. This was later relaxed in \cite[section 9]{CK07} and \cite[Theorem 3.3]{CZ14}.

The following theorem is an instance of weak duality for linear programs. It appears in \cite{Co} and a version of the same linear program, expressed in terms of pair-correlation functions, was studied by Torquato and Stillinger \cite{TS}. Depending on the space of $f$ one considers for the Cohn--Elkies bound, the dual object $\mu$ will be one kind of generalized function or another. We suppose $f$ varies over continuous functions, so that $\mu$ varies over measures.
\begin{theorem} \label{thm:dual}
If $\mu = \delta_0 + \nu$ is a measure on $\R^{D}$ with $\nu \geq 0$ supported in $\{ x \in \R^D \ ; \  |x| \geq r \}$ and $\widehat{\mu} \geq c \delta_0$, then the Cohn--Elkies linear programming bound for center density in $\R^D$ is at least
\begin{equation} \label{eqn:dual}
c \left( \frac{r}{2} \right)^D.
\end{equation}
\end{theorem}

\begin{proof}
The duality at work here is the Plancherel formula. Let $f$ and $\mu$ be as in the Cohn--Elkies program and its dual.
We can assume it is the same $r$ for both $f$ and $\mu$, by scaling if necessary. Indeed, replacing $f(x)$ by $sf(s'x)$ for $s,s' > 0$ does not change (\ref{eqn:ce}).
Then $\langle f, \mu \rangle \leq f(0)$ because $f \leq 0$ on the support of $\mu - \delta_0$. By Plancherel, since $\widehat{\mu} \geq c \delta_0$ and $\widehat{f} \geq 0$,
\begin{equation} \label{eqn:plancherel}
f(0) \geq \langle f, \mu \rangle = \langle \widehat{f}, \widehat{\mu} \rangle \geq c\widehat{f}(0)
\end{equation}
so $f(0)/\widehat{f}(0) \geq c$. Since this holds for any $f$ satisfying the Cohn--Elkies constraints, the LP bound is at least $c (r/2)^D$. 
\end{proof}

In other words, the LP bound thinks there is a packing of center density $c(r/2)^D$, even if there is only a measure imitating such a packing. For comparison, given a lattice $\Lambda$ with center density $\rho$ and dual lattice $\Lambda^*$, the following measure has transform given by Poisson summation:
\[
\mu = \sum_{x \in \Lambda} \delta_x, \quad \widehat{\mu} = \rho \sum_{\xi \in \Lambda^*} \delta_{\xi}
\]
The factor $c$ in this case is exactly $\rho$, and the result from Theorem~\ref{thm:dual} is the same as one would find by comparing the LP bound to the packing with spheres centered at the points of $\Lambda$ (here, we scale so that the spheres in the packing have radius 1).
In this sense, the $\mu$ in the ``dual" program has a more direct connection to the underlying packings, compared to the Cohn--Elkies function $f$. In the sharp cases in 8 or 24 dimensions, the packing is related to the zeros of $f$ and $\widehat{f}$, and $f$ can be recovered from these zeros by interpolation \cite{CKMRV, V}.

Figure~\ref{fig:spikes} compares our candidate $\mu$ to a numerical approximation of the optimal $f$. 
The proof of Theorem~\ref{thm:dual} shows that, for $f$ and $\mu$ to give optimal bounds, $\mu - \delta_0$ should be supported on the zeros of $f$. That way, there is no loss in the inequality $f(0) \geq \langle f, \mu \rangle$. Similarly, $\widehat{\mu}$ should be supported on the zeros of $\widehat{f}$ together with the origin. In our example, the support is larger, but aligns well with the initial zeros. The function plotted in Figure~\ref{fig:spikes} approximates $f$ by a polynomial of degree 24 times a Gaussian $\exp(-\pi x)$, written as a function of squared distance. 
It was computed in Julia using sum-of-squares techniques implemented by de Laat and Leijenhorst in \cite{dLL}. For visibility, we plot only the polynomial factor without the Gaussian, which of course vanishes at the same points as $f$.

\begin{figure}
\includegraphics[width=0.75\textwidth]{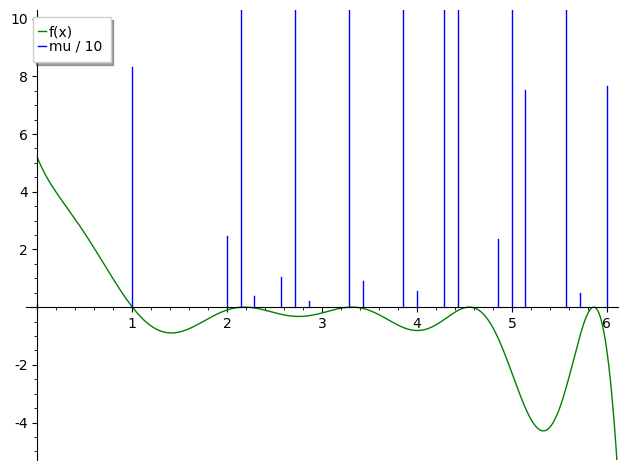}
\includegraphics[width=0.75\textwidth]{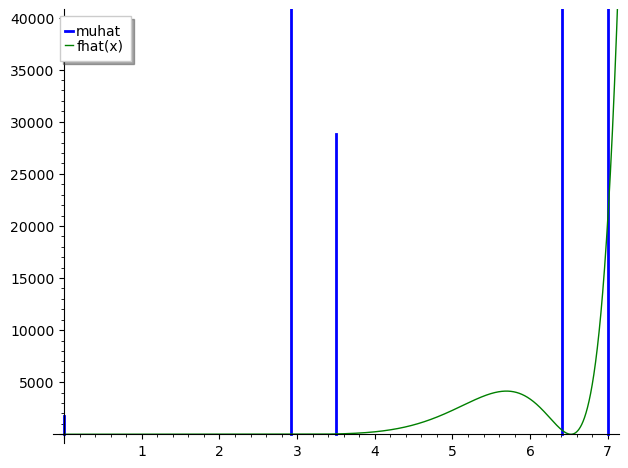}
\caption{
Top: 
a numerical approximation to the optimal Cohn--Elkies function $f$ in green, compared to the measure $\mu = \sum_n a_n \delta_{\sqrt{n}}$ in blue, as functions of squared distance $x$. We rescale $\mu$ by $n \mapsto n/7$ so that the first spike coincides with the sign change of $f$ at $x=1$. 
Bottom: $\widehat{\mu} = (2/\sqrt{N})^k \sum_n b_n \delta_{2\sqrt{n/N}}$ in blue compared to the polynomial part of $\widehat{f}$ in green.
With $N=48$, and scaling the Fourier transform reciprocally to $\mu$ by $n \mapsto 7n$, the spikes of $\widehat{\mu}$ occur at $x= 7n/12$. 
For visibility, without affecting the locations of zeros or spikes, the plots show only the polynomials $f(x) \exp(\pi x)$ and $\widehat{f}(x) \exp(\pi x)$ without the Gaussian factor, and large values of $a_n$ or $b_n$ are cut off.
The first two zeros of both $f$ and $\widehat{f}$ align closely with the supports of $\mu$ and $\widehat{\mu}$. 
}
\label{fig:spikes}
\end{figure}

To make a $\mu$ satisfying the constraints of Theorem~\ref{thm:dual}, Cohn and Triantafillou \cite{CT} proposed
\[
\mu = \sum_{n=0}^{\infty} a_n \delta_{\sqrt{n}}
\]
where $a_n$ are the Fourier coefficients of a modular form.
They assumed that the dimension $D$ is a multiple of 4, and considered modular forms of weight $D/2$ with trivial character on $\Gamma_0(N)$. Suitable values of $N$ gave results in dimensions 12 and 16, showing that the linear programming bound does not match the densities of the presumed best packings. However, these dimensions were too high to make a favourable comparison with semidefinite methods at the time. The possibility of a better packing remained, which could even match the LP bound, until \cite{CdLS} gave a strong enough SDP bound in dimensions 12 and 16 to show that the earlier bounds from \cite{CT} imply non-sharpness in those cases.
In dimension 6, we apply a similar strategy, showing decisively that the linear programming bound is not sharp by comparing it to Theorem~\ref{thm:sdp}. 
Theorem~\ref{thm:main} follows from the existence of a suitable modular form. Since 6 is not a multiple of 4, we must work with modular forms of odd weight 3, using non-trivial Dirichlet characters. In particular, we used the quadratic characters $\chi_3$ and $\chi_4$ defined by $-1 \bmod 3 \mapsto -1$ and $-1 \bmod 4 \mapsto -1$ respectively.
For each of these there are spaces of modular forms denoted by $M_k(\Gamma_0(N),\chi)$, which we will discuss further in Sections~\ref{sec:ct} and \ref{sec:space}.

\begin{proposition} \label{prop:mf}
There is a modular form
\[
g \in M_3(\Gamma_0(48),\chi_3) \oplus M_3(\Gamma_0(48),\chi_4)
\]
whose Fourier expansion and that of the transform $\widetilde{g}(z) = -i 48^{-3/2} z^{-3}g(-1/(48z))$ satisfy
\[
g(z) = \sum_{n=0}^{\infty} a_n e^{2\pi inz}, \quad \widetilde{g}(z) = \sum_{n=0}^{\infty} b_n e^{2\pi inz}
\]
where $a_n \geq 0$ and $b_n \geq 0$ for all $n$, $a_0=1$, $b_0 \geq 0.6168035$, and $a_n = 0$ for $1 \leq n \leq 6$. Moreover, $a_n = 0$ for all $n \equiv 1 \bmod 4$.
\end{proposition}
Again, this is a rounded version of a more precise statement: the $g$ we construct has $b_0$ equal to the value from (\ref{exact-quadratic}) multiplied by $\frac{768}{343}\sqrt{12}$, via (\ref{eqn:bound}) below.

\begin{figure}
\centering
\includegraphics[width=0.475\textwidth]{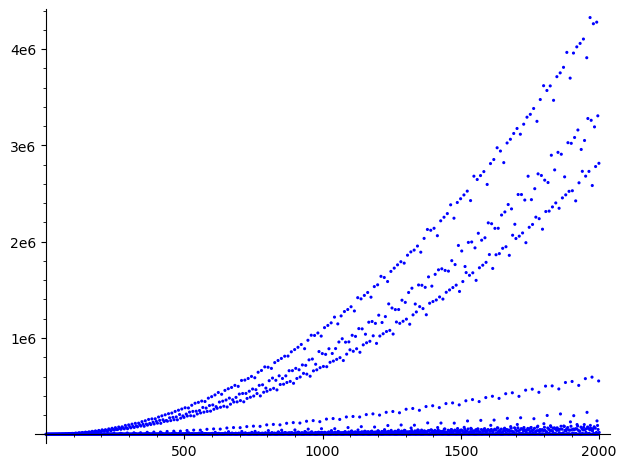}
\includegraphics[width=0.475\textwidth]{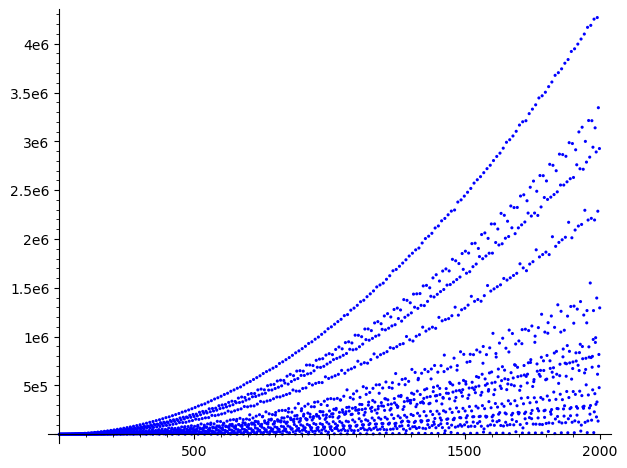} \\
\includegraphics[width=0.475\textwidth]{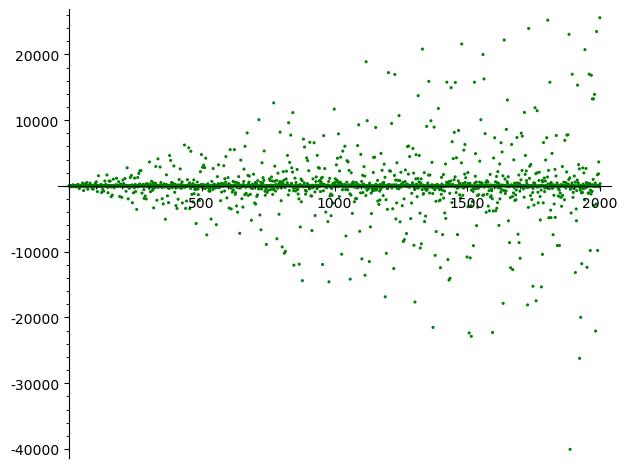}
\includegraphics[width=0.475\textwidth]{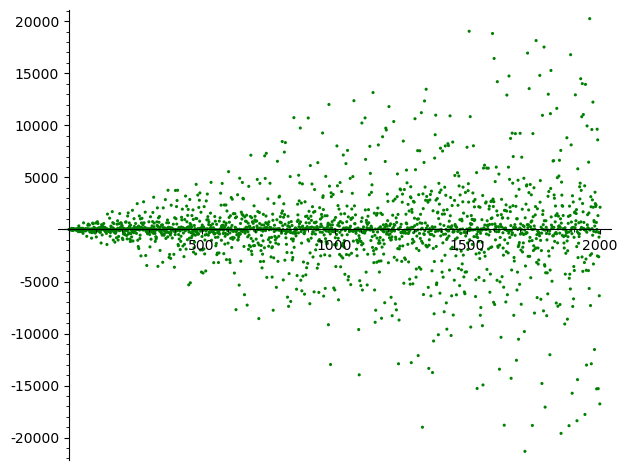}
\caption{
Contributions to $a_n$ (left) and $b_n$ (right) from Eisenstein series (top) and cuspforms (bottom) for $n \leq 2000$.
Both contributions to $a_n$ vanish when $n \equiv 1 \bmod 4$.
Otherwise, the Eisenstein part has different rates of quadratic growth depending on the factors of $n$, which we analyze in Section~\ref{sec:eisenstein}.
The cuspidal part stays within a linear envelope unless $n$ has many factors, by Deligne's bound Theorem~\ref{thm:cusp-bound}.
}
\label{fig:plots}
\end{figure}

The direct sum in Proposition~\ref{prop:mf} is a 44-dimensional space of modular forms, in which we find $g$ by solving a linear program that forces a finite number of the constraints $a_n \geq 0$ and $b_n \geq 0$. To verify the remaining inequalities as $n \rightarrow \infty$, we decompose $g$ as a sum of Eisenstein series and cuspforms. The Eisenstein contribution to $a_n$ or $b_n$ grows quadratically with $n$, whereas the cuspidal part has absolute value at most $n^{1+o(1)}$, from which one deduces the inequalities for large $n$. The different rates of growth are illustrated in Figure~\ref{fig:plots}.
To show the exact equalities $a_n = 0$ for $n \equiv 1 \bmod 4$, we must argue more carefully using the action of Hecke operators and quadratic twists. 
The fact that $a_n=0$ for $n \equiv 1 \bmod 4$ is visible in Figure~\ref{fig:spikes}, where, so to speak, every fourth spike seems to be missing. By increasing $N$ and using more characters instead of only $\chi_3$ and $\chi_4$, one could try to remove even more spikes and obtain a better match with the zeros of $f$.

There are other approaches to rigorously estimating the LP bound. Li \cite{L} discretizes Theorem~\ref{thm:cohn-elkies}, replacing $\R^D$ by a grid, and using duality for the resulting finite-dimensional linear programs rather than Theorem~\ref{thm:dual}. This proves that the Cohn--Elkies bound is not sharp in dimensions 3, 4, and 5, and gives bounds in all dimensions $3 \leq D \leq 13$ (ignoring the sharp case $D=8$). In dimension 6, the bound from \cite{L} is 0.07632412 in terms of center density, whereas one would need roughly 0.0794 to conclude from Theorem~\ref{thm:sdp} that the LP bound is not sharp. However, we imagine the only obstacle is computational power (and the current value would already suffice if one knew that $E_6$ is a densest packing, or if the SDP bound could be improved enough). 
Torquato and Stillinger \cite{TS}, motivated by hyperuniformity and the idea that disordered sphere packings in high dimensions could be much denser than known examples based on lattices, analyzed a choice of $\mu$ that gives the best result currently known asymptotically as $D \rightarrow \infty$. Their result, in terms of density rather than center density, shows that the LP bound is at least $2^{-\tau D + o(D)}$ where
\[
\tau = \frac{3 - 1/\log{2}}{2} \approx 0.7786
\]
and $o(D)$ denotes sublinear terms, meaning that $o(D)/D \rightarrow 0 $ as $D \rightarrow \infty$.
Numerical extrapolations from \cite{ACHLT} suggest the LP bound is $2^{-\lambda D + o(D) }$ where $\lambda \approx 0.6044$ and the authors speculate on a closed form for $\lambda$. However, the best proven asymptotic upper bound remains that of Kabatyanskii and Levenshtein \cite{KL}: the LP bound is at most $2^{-\kappa D + o(D) }$ where $\kappa \approx 0.59905576$ is the root of an explicit equation. The best known packings have density $2^{-D + o(D)}$. To our knowledge, there is no rigorous analysis of the SDP bound as $D \rightarrow \infty$.
That would be an important step in showing that there are only finitely many dimensions where the linear programming bound is sharp, and perhaps only 1, 2, 8, 24.

In Section~\ref{sec:ct}, we review the method of \cite{CT}, giving an alternative calculation of the main Fourier transform underlying the approach.
In Section~\ref{sec:space}, we describe the optimization space from Proposition~\ref{prop:mf} in more detail.
In Section~\ref{sec:choice}, we explain how to find the solution using exact arithmetic in the field $\Q(\sqrt{3})$.
In Sections~\ref{sec:cusp-total} to \ref{sec:eisenstein}, we estimate at what point the Eisenstein part dominates the cuspidal part, and explain the vanishing $a_n = 0$ for $n \equiv 1 \bmod 4$ in Section~\ref{sec:1mod4}.
Supplementary files available at \url{https://arxiv.org/abs/2211.09044} explain these calculations, and the final verification that $a_n, b_n \geq 0$, in more detail.
In Section~\ref{sec:d6-z6}, we prove an identity (\ref{eqn:theta-identity}) between theta series for the lattices $D_6$, $D_6^*$, and $\Z^6$. We had expected these three to be independent elements of the optimization space, and the linear relation countering that expectation is specifically 6-dimensional.

\begin{table}
\begin{tabular}{ccccc}
$D$ & Record density & SDP bound &  \multicolumn{2}{c}{LP bound}  \\ 
& & & Dual bound & Upper bound \\ \hline
1 & 1 & & & 1 \\
2 & 0.906899 & & & 0.906899? \\
\cellcolor{lightgray!40!white}3 & 0.740480 & 0.770271 & 0.770657 & 0.779747 \\
\cellcolor{lightgray!40!white}4 & 0.616850 & 0.636108 & 0.637303 & 0.647705 \\
\cellcolor{lightgray!40!white}5 & 0.465257 & 0.512646 & 0.517236 & 0.524981 \\
\cellcolor{lightgray!40!white}6 & 0.372947 & 0.410304 & {\bf 0.410948} & 0.417674 \\
7 & 0.295297 & 0.321148 & 0.301191 & 0.327456 \\
8 & 0.253669 &  &  & 0.253669 \\
9 & 0.145774 & 0.191121 & 0.164925 & 0.194556 \\
10 & 0.099615 & 0.143411 & 0.106256 & 0.147954 \\
11 & 0.066238 & 0.106726 & 0.078504 & 0.111691 \\
\cellcolor{lightgray!40!white}12 & 0.049454 & 0.079712 & 0.083381 & 0.083776 \\
13 & 0.032014 & 0.060165 & 0.032522 & 0.062482 \\
14 & 0.021624 & 0.045062 & & 0.046365 \\
15 & 0.016857 & 0.033757 & & 0.034249 \\
\cellcolor{lightgray!40!white}16 & 0.014708 & 0.023995 & 0.025011& 0.025195 \\
24 & 0.001929 & & & 0.001929
\end{tabular}
\caption{The LP bound exceeds the SDP bound in dimensions $D=3, 4, 5, 6, 12, 16$, shown in grey. The LP bound is sharp for $D=1,8,24$ and conjecturally $D=2$. Values of the Cohn--de Laat--Salmon SDP bound are taken from \cite[Table 1.1]{CdLS}, as are the highest known densities and the numerical LP bounds computed in \cite{ACHLT}. Li's dual bound on the LP bound from \cite{L} for $D \leq 13$, and the values for $D=12, 16$ from Cohn--Triantafillou \cite[Table 6.1]{CT}, have been converted from center density to density.
The value from Theorem~\ref{thm:main} is shown in bold.}
\label{table:sitch}
\end{table}

\section{The method of Cohn--Triantafillou} \label{sec:ct}

The key Fourier calculation behind our proof is given in the following proposition. Cohn and Triantafillou proved this in \cite[Proposition 2.2]{CT} assuming that $k$ is an integer (i.e. the dimension is even), and that $g$ is modular for a congruence group $\Gamma_1(N)$. We give another proof without those assumptions: it is enough for $g$ to be modular for a smaller group given by generators in (\ref{eqn:both-periodic}), where $N > 0$ could take non-integer values. For $k$ a half-integer and $\Im(z) > 0$, define $(-iz)^{-k}$ by the standard branch satisfying $1^{-k}=1$.
Write $\mathbb{H}$ for the half-plane $\{ z \in \C \ ; \ \Im(z) > 0 \}$.
\begin{proposition} \label{prop:fourier}
For a function $g: \mathbb{H} \rightarrow \C$, $k = D/2$ an integer or half-integer, and $N >0$, let
\begin{equation} \label{eqn:transform}
\widetilde{g}(z) = (-iz)^{-k} N^{-k/2} g\left( \frac{-1}{Nz} \right).
\end{equation}
Suppose $g$ and $\widetilde{g}$ are both periodic, with Fourier expansions
\[
g(z) = \sum_{n=0}^{\infty} a_n e^{2\pi i nz} \qquad \text{and} \qquad \widetilde{g}(z) = \sum_{n=0}^{\infty} b_n e^{2\pi i n z}.
\]
Then
\[
\sum_{n=0}^{\infty} a_n \delta_{\sqrt{n}} \qquad \text{and} \qquad (2/\sqrt{N})^k \sum_{n=0}^{\infty} b_n \delta_{2\sqrt{n/N}}
\]
are Fourier transforms of each other as tempered distributions on $\R^D$, where $\delta_r$ denotes a spherical delta at radius $r$.
\end{proposition}

\begin{proof}
By definition, the claim is that
\[
\sum_{n=0}^{\infty} a_n f(\sqrt{n}) = (2/\sqrt{N})^k \sum_{n=0}^{\infty} b_n \widehat{f}\left( 2\sqrt{\frac{n}{N} } \right)
\]
for all radial Schwartz functions $f$. 
By approximation, it is enough to check this for Gaussians $f(x) = e^{\pi i  |x|^2 z}$ with $\Im(z)>0$.
The Fourier transform is then
\[
\widehat{f}(\xi) = (i/z)^{k} e^{\pi i |\xi|^2 (-1/z) }
\]
On the left,
\[
\sum_n a_n f(\sqrt{n}) = \sum_n a_n e^{\pi i n z} = g(z/2).
\]
On the right,
\[
\sum_n b_n \widehat{f}(2 \sqrt{n/N}) = \sum_n b_n e^{-4\pi i n /(Nz) }  (i/z)^{k} = \widetilde{g}\left( \frac{-1}{Nz/2} \right)  (i/z)^{k}
\]
Note that
\[
\widetilde{g}\left(\frac{-1}{Nz/2} \right) = \left(-i \left( \frac{-1}{Nz/2} \right) \right)^{-k} N^{-k/2} g(z/2) =2^{-k} N^{k/2} (i/z)^{-k} g(z/2)
\]
So, as required, $(\text{right})=2^{-k} N^{k/2} (\text{left})$.
\end{proof}

The periodicity of both $g$ and $\widetilde{g}$ requires $g$ to transform as a modular form, and $\widetilde{g}$ is an example of an Atkin--Lehner involution applied to $g$.
Let
\begin{equation} \label{eqn:win}
w_N = \begin{pmatrix} 0 & -1 \\ N & 0 \end{pmatrix}, \qquad w_N(z) = \frac{-1}{Nz}
\end{equation}
Then, for integer $k$,
\begin{equation} \label{eqn:transform-slash}
\widetilde{g} = (-i)^{-k} g| w_N
\end{equation}
where the \emph{slash action} is given by
\begin{equation} \label{eqn:slash}
g \big\rvert \begin{pmatrix} a & b \\ c & d \end{pmatrix} (z) = (ad-bc)^{k/2} (cz+d)^{-k}  g\left( \frac{az+b}{cz+d} \right) 
\end{equation}
and satisfies a chain rule $g \big\rvert \gamma_1 \gamma_2 = (g \big\rvert \gamma_1) \big\rvert \gamma_2$.
The factor $(-i)^{-k}$, or simply $i^k$, is the correct sign so that both sequences $a_n, b_n$ can be non-negative, as one can see by taking $z=it$ on the imaginary axis in (\ref{eqn:transform}).
The periodicity of $\widetilde{g}$ amounts to $g | w_N T = g|w_N$, where $T(z) = z+1$ is the standard translation.
Equivalently, $g | w_N T w_N^{-1} = g$.
Thus the periodicity of $g$ and $\widetilde{g}$ requires two modular properties for $g$:
\begin{equation} \label{eqn:both-periodic}
g, \widetilde{g} \ \text{ both periodic} \Longleftrightarrow g = g|T = g|w_N T w_N^{-1}
\end{equation}
In matrix form, the condition is that $g$ must be modular (of weight $k=D/2$) with respect to the group generated by $\begin{pmatrix} 1 & 1 \\ 0 & 1 \end{pmatrix}$ and $\begin{pmatrix} 1 & 0 \\ N & 1 \end{pmatrix}$. 

If $N > 4$, then there is an infinite-dimensional space of candidates $g$ transforming as per (\ref{eqn:both-periodic}), since a scaling $z \mapsto \sqrt{N} z$ conjugates these generators into a Hecke triangle group \cite[Chapter 4]{BK}.
However, these are not as explicit as modular forms on $\Gamma_1(N)$.
Following \cite{CT}, we will assume $N$ is an integer and impose modularity under a larger group to obtain a finite-dimensional space that is amenable to computation.

The action of $w_N$ by conjugation is
\begin{equation} \label{eqn:w-conjugation}
w_N \begin{pmatrix} a & b \\ c & d \end{pmatrix} w_N^{-1} = \begin{pmatrix} d & -c/N \\ -Nb & a \end{pmatrix}
\end{equation}
In particular, $w_N$ normalizes the congruence subgroups $\Gamma_0(N)$ and $\Gamma_1(N)$. Recall that these are the subgroups of $\SL(2,\Z)$ defined by
\[
 \Gamma_0(N) = \begin{pmatrix} * & * \\ 0 & * \end{pmatrix} \bmod N,  \quad \Gamma_1(N) = \begin{pmatrix} 1 & * \\ 0 & 1 \end{pmatrix} \bmod N
\]
which are clearly normalized by $w_N$, in view of (\ref{eqn:w-conjugation}).
It follows that if $g$ is a modular form for $\Gamma_1(N)$, then both $g, \widetilde{g}$ are periodic.
That was the case considered in \cite{CT}. It is convenient for computational purposes because of existing software developed in GP/Pari \cite{pari}, Sage \cite{sage}, and Magma \cite{magma}. 

The space of modular forms for $\Gamma_1(N)$ decomposes into spaces for $\Gamma_0(N)$ with various Dirichlet characters modulo $N$ (see, for instance, \cite[p. 137]{K}). 
Recall that, by definition, modular forms with character $\chi$ are holomorphic functions with a growth condition and satisfying the functional equations
\[
f\left( \frac{az+b}{cz+d} \right) = \chi(d) (cz+d)^k f(z)
\]
whenever $a,b,c, d$ are integers obeying $ad-bc=1$ and $c$ is divisible by $N$, that is, $\begin{pmatrix} a & b \\ c & d \end{pmatrix} \in \Gamma_0(N)$. 
Such $f$ form a vector space over $\C$ which we denote by $M_k(\Gamma_0(N),\chi)$. Calculations with this space have been implemented in Pari using the command \texttt{mfinit} which takes as input $k$, $N$, and various possible descriptions of $\chi$. 

Cohn and Triantafillou stated their results in terms of $\Gamma_1(N)$, but used only the trivial character, which was already enough to obtain interesting bounds with $k=6$ or $k=8$ for the dimensions they considered.
In case $k$ is odd, we must use $\Gamma_1(N)$, rather than $\Gamma_0(N)$ with trivial character, and decompose with respect to characters modulo $N$ satisfying $\chi(-1)=-1$. One must have $\chi(-1)=(-1)^k$ for there to be any non-zero modular forms of character $\chi$ on $\Gamma_0(N)$, in view of the action of $-I: z \mapsto (-z)/(-1)=z$. 

Taking the measure $\mu$ in Theorem~\ref{thm:dual} to be a sum of point masses as in Proposition~\ref{prop:fourier} led Cohn and Triantafillou to the following, which can be optimized over a finite-dimensional space of modular forms.
\begin{proposition}[Cohn--Triantafillou] \label{prop:bound}
Suppose $g$ is a modular form of weight $D/2$ for $\Gamma_1(N)$ and the Fourier coefficients of $g = \sum_n a_n q^n $ and $\widetilde{g} = \sum_n b_n q^n$ obey the following inequalities, where $q=e^{2\pi i z}$:
\begin{align*}
a_n &\geq 0, b_n \geq 0 \\
a_0 &=1, b_0 > 0 \\
a_n &= 0 \qquad \text{for} \qquad 1 \leq n < T
\end{align*}
Then the linear programming bound in $\R^{D}$ is at least
\begin{equation} \label{eqn:bound}
b_0 \left( \frac{2}{\sqrt{N} } \right)^{D/2} \left( \frac{\sqrt{T}}{2} \right)^{D}
\end{equation}
\end{proposition}

This is a linear program with a finite number of variables, as many as the dimension of the space of modular forms of weight $D/2$ for $\Gamma_1(N)$. 
However, it involves infinitely many constraints $a_n \geq 0$, $b_n \geq 0$. 
In practice, one solves a linear program enforcing only a finite number of these. One must then separate the solution $g$ into a cuspidal part and an Eisenstein part. The remaining inequalities for large $n$ are established by showing that the Eisenstein part is eventually positive, and large enough that the cuspidal part can be neglected.

The cuspidal part can be bounded using Deligne's theorem on the Ramanujan conjecture for holomorphic forms. For a statement without proof, one can refer to \cite[eq. (14.54)]{IK}. The latter follows from Deligne's work on the Riemann Hypothesis for varieties over finite fields \cite[Th\'eor\`eme 8.2]{D}.
See also \cite{D-bourbaki} and the references therein to earlier work \cite{Ihara, KS} linking the Fourier coefficients and Hecke eigenvalues to the number of points on associated varieties.
\begin{theorem}[Deligne] \label{thm:cusp-bound}
If $k \geq 2$ is an integer, and $f$ is a Hecke cuspform of weight $k$ for $\Gamma_0(N)$ with character $\chi$, normalized so that 
\[
f(z) = \sum_{n>0} c_f(n) e(nz) = 1 \cdot e(n_{{\rm min}} z) + \ldots \hspace{1cm} (e(t) = e^{2\pi i t})
\]
with leading coefficient $1$, then for all $n$,
\[
|c_f(n)| \leq n^{(k-1)/2} \sigma_0(n)
\]
where $\sigma_0(n)$ is the number of positive divisors of $n$.
\end{theorem}
The corresponding result for $k=1$ is also true, by a theorem of Deligne and Serre \cite{DS}, but that case corresponds to dimension 2, where the LP bound is expected to be sharp, and where both exponents $k-1$ and $(k-1)/2$ equal 0. For us, $k=3$.

The divisor function $\sigma_0(n)$ grows slower than any power of $n$, but can exceed any fixed power of $\log{n}$. At its largest, it behaves roughly like 
\[
2^{ \log{n} \ \div \ \log\log{n} }
\]
as discussed for instance in \cite[Theorem 317]{HW}. We quote a sharp version of this from \cite{NR}.
\begin{theorem}[Nicolas, Robin] \label{thm:max-divisors}
The maximum
\[
\max_{n \geq 2} \quad \log{\sigma_0(n) } \frac{ \log\log{n} }{(\log{2})( \log{n}) }
\]
is attained at
\[
n=2^5 \cdot 3^3 \cdot 5^2 \cdot 7 \cdot 11 \cdot 13 \cdot 17 \cdot 19 = 6 \ 983 \ 776 \ 800
\]
with value $1.5379 \ldots$
\end{theorem}
This number 1.5379 is larger than the limiting value 1 that one could take as $n \rightarrow \infty$, but in our application, it is helpful if one can show that the Eisenstein part dominates the cuspidal part even for relatively small values of $n$.

Write $a_n$ as the sum of the contributions from Eisenstein series and from cuspforms: 
\begin{equation} \label{eqn:eis+cusp}
a_n = a_{n,{\rm eis}} + a_{n,{\rm cusp}}
\end{equation}
and similarly for $b_n$. 
It can happen that $a_n$ vanishes along an arithmetic progression (for instance, we will have $a_n=0$ for all $n \equiv 1 \bmod 4$). For the remaining $n$, we will show $a_n \geq 0$ by the following inequalities.
Using explicit formulas for the Eisenstein series, one can show a bound of the form
\begin{equation} \label{eqn:eps-def}
a_{n,{\rm eis}} \geq \varepsilon n^{k-1}
\end{equation}
whereas Deligne's theorem implies
\begin{equation} \label{eqn:C-def}
|a_{n,{\rm cusp}}|  \leq C n^{(k-1)/2} \sigma_0(n)
\end{equation}
where $\varepsilon > 0$ and $C > 0$ are constants (independent of $n$) determined by the choice of variables $x_j$ in Proposition~\ref{prop:bound}. 
These bounds entail
\[
a_n \geq a_{n,{\rm eis}} - |a_{n,{\rm cusp}}| \geq \varepsilon n^{k-1} - C n^{(k-1)/2} \sigma_0(n).
\]
This will imply $a_n \geq 0$ once $n$ is large enough that
\begin{equation} \label{eqn:divisors-versus-n}
\sigma_0(n) \leq \frac{\varepsilon}{C} n^{(k-1)/2} 
\end{equation}
After logarithms, the comparison (\ref{eqn:divisors-versus-n}) becomes
\begin{equation} \label{eqn:divisors-versus-logs}
\log \sigma_0(n) + \log{\frac{C}{\varepsilon}} \leq \frac{k-1}{2} \log{n}
\end{equation}
Suppose 
\begin{equation} \label{eqn:divisors-exprate}
 \log{\sigma_0(n)} \leq R \frac{\log{n}}{\log\log{n}}
\end{equation}
where one can take $R = 1.5379 \log{2}$ as in Theorem~\ref{thm:max-divisors}, or take $R$ arbitrarily close to $\log{2}$ by assuming $n$ is large enough.
The criterion (\ref{eqn:divisors-versus-logs}) is satisfied provided that
\begin{equation} \label{eqn:satisfied}
\frac{R}{\log\log{n}} + \frac{\log{(C/\varepsilon)}}{\log{n}} \leq \frac{k-1}{2}
\end{equation}
which holds for all large enough $n$. Explicitly, if
\[
n \geq \max\left( \exp \left( \frac{4 \log(C/\varepsilon)}{k-1} \right), \exp\exp \left(\frac{4R}{k-1} \right) \right)
\]
then (\ref{eqn:satisfied}) is satisfied because each term on the left is at most half the right. 
For instance, it is enough to take $n \geq \max( (C/\varepsilon)^2, 4590)$ when $k=3$ and $R$ is taken from Theorem~\ref{thm:max-divisors}. 

In practice, one can do better by solving (\ref{eqn:satisfied}) numerically. The estimates from Sections~\ref{sec:cusp-total} and \ref{sec:eisenstein} will show that $C=21.6161$ and $\varepsilon = 8.7536 \times 10^{-6}$ are admissible for $a_n$, while $C=24.0265$ and $\varepsilon = 0.001358$ are admissible for $b_n$. 
It follows that $a_n \geq 0$ for $n \geq 5347177639 \approx 5 \times 10^9$, while $b_n \geq 0$ for $n \geq 8126856 \approx 8 \times 10^6$. 

The verification that $a_n, b_n \geq 0$ for the remaining values of $n$ can be done on a personal computer running Pari. The calculation takes a matter of hours for $b_n$, or a matter of days for $a_n$. 
It seemed too memory-intensive to compute all the values of $a_{n,{\rm cusp}}$ for $n \leq n_0 =5347177638$, whereas $a_{n,{\rm eis}}$ and $\sigma_0(n)$ are given by elementary formulas in terms of the factors of $n$.
We first compute $a_{n,{\rm eis}}$ for all $n \leq n_0$, and check whether $a_{n,{\rm eis}} - C n \sigma_0(n) \geq 0$, in which case positivity follows from Theorem~\ref{thm:cusp-bound}. If not, we add $n$ to a list of indices for which we will compute $a_{n,{\rm cusp}}$. 
In the more difficult case of $a_n$, this list contains 67250 values, culminating in $n=261777516$.
For $b_n$, we only needed 4329 values of $b_{n,{\rm cusp}}$, the last being $n=556738$. 
The values can be found in the supplementary files \texttt{checklist-an-67250.txt} and \texttt{checklist-bn-4329.txt} at \url{https://arxiv.org/abs/2211.09044}.
It took a week to make the list for $a_n$, and two days to confirm that $a_{n,{\rm eis}}+a_{n,{\rm cusp}} \geq 0$ for each of these $n$. No doubt some improvements could be made to our implementation, and in particular it would have been possible to check different ranges in parallel.

The dependence on $k$ in (\ref{eqn:satisfied}) was a challenge making our six-dimensional example more difficult in some ways than the higher-dimensional ones from \cite{CT}. When $k$ is large, an Eisenstein part of order $n^{k-1}$ will very quickly dominate a cuspidal part of order $n^{(k-1)/2}$ (although of course there could be difficulties for any given $k$ if the ratio $C/\varepsilon$ is too large).
When $k$ is only 3, we encountered several examples where $n$ must be quite large before the asymptotics make themselves felt. For instance, with $T=7$ and $N=48$ in Proposition~\ref{prop:bound}, forcing the inequalities for $n \leq 2000$ might have seemed more than sufficient, but it failed to make $a_n,b_n \geq 0$ for larger $n$. Choosing $x_j'$ according to that smaller linear program, instead of the solution $x_j$ used here, leads to an infeasible candidate where some of the corresponding coefficients $a_n',b_n'$ are negative. Indeed, in that case, $b_{2890}' < 0$. 
For various pairs $(N,T)$, some of our candidates seemed valid for tens of thousands of coefficients, only to fail at large values of $n$ with a special factorization for which the Eisenstein contribution was negative.
The role of the factors is partly explained by Proposition~\ref{prop:ratio-4-3}.

\section{Space of modular forms} \label{sec:space}

In this section, we describe the space of candidates $g$ over which we optimize, and indicate how it relates to the presumed best packing $E_6$. We will take $g = \sum_j x_j f_j$ to be a linear combination of 44 modular forms $f_1, \ldots, f_{44}$. The quantity $b_0$, which is to be maximized as in Proposition~\ref{prop:bound}, will take the shape
\begin{align}
b_0 = &\frac{64}{27}x_1 + \frac{8}{27}x_2 + \frac{1}{27} x_3 + \frac{1}{216}x_{4} + \frac{1}{1728}x_5 + \label{b0} \\
&\frac{1}{\sqrt{12}} \left( 9x_{23} + \frac{9}{8}x_{24} + \frac{1}{3} x_{25} + \frac{9}{64}x_{26} + \frac{1}{24}x_{27} + \frac{1}{192}x_{28}\right) \nonumber
\end{align}
The optimal coefficients $x_j$ lie in $\Q(\sqrt{3})$ because of the factor $\sqrt{12}$ arising here. Table~\ref{table:numbers} shows their values truncated to 4 decimal places for readability, and we will discuss how to find them exactly in Section~\ref{sec:choice}. In this section, we describe the basis functions $f_j$ and their transforms.

We follow the approach from Proposition~\ref{prop:bound}, taking 
\begin{equation} \label{eqn:choice-of-N}
N=48 = 2^4 \cdot 3
\end{equation}
The weight relevant to 6-dimensional packing is $k=D/2=3$. 

We used only two characters $\chi_3$ and $\chi_4$, induced from the characters mod 3 and mod 4 given by $-1 \bmod 3 \mapsto -1$ and $-1 \bmod 4 \mapsto -1$. In principle, one could do even better using more characters, but these were already enough to prove Theorem~\ref{thm:main}. We choose $T=7$ in Proposition~\ref{prop:bound}, that is, we seek a solution of the form
\[
g(z) = 1 + a_7 q^7 + \ldots \quad (q = e^{2\pi i z} )
\]
in the direct sum
\[
M_3(\Gamma_0(48), \chi_3) \oplus M_3(\Gamma_0(48), \chi_4).
\]

The summand for $\chi_3$ contains the theta series of the lattice $E_6$, which is presumed to give a sphere packing of maximal density. 
In terms of the basis $f_1, \ldots, f_{44}$ that we will use, this theta series is expressed as
\begin{equation} \label{eqn:e6}
\vartheta_{E_6}(z) = 81 f_1(z) - 9f_6(z) = 1 + 72q + 270q^2 + 936q^3 + 2160q^4 + \ldots 
\end{equation}
This motivated our choice of characters. From the outset, the summand $M_3(48,\chi_3)$ in the optimization space gives enough flexibility to at least recover the density of the (presumed) optimal packing.
The dual lattice $E_6^*$, suitably scaled, has theta series
\[
9f_1(z) - 9f_6(z) = 1 + 0q + 54q^2 + 72q^3 + 0q^4 + 432q^5 + 270q^6 + 0 q^7 + 918q^8 + 720q^9 + \ldots
\]
The other subspace for character $\chi_4$ contains the theta series of the lattices $D_6$ and $\Z^6$, namely
\begin{align*}
\vartheta_{D_6}(z) &= 64f_{23}(z) - 4f_{29}(z) = 1 + 60q + 252q^2 + 544q^3 + 1020q^4 + \ldots \\ 
\vartheta_{\Z^6}(z) &= 16f_{23}(z) - 4f_{29}(z) = 1+ 12q + 60q^2 + 160q^3 + 252q^4 + \ldots 
\end{align*}
The space also contains theta series for the dual lattices, but as we explain in Section~\ref{sec:d6-z6}, there is a linear relation beween those of $D_6$, $D_6^*$, and $\Z^6$.
We increased the level $N$ along multiples of 12 so that all these theta series would be available, and tried several values of $T$ at each level, stopping at the pair $(N=48,T=7)$ where the resulting bound was strong enough to prove Corollary~\ref{cor:cor}.
Theta series are convenient because they automatically satisfy the inequalities from Proposition~\ref{prop:bound}.

An important feature of the choice of level $N=48$ is that it allows certain quadratic twists. Recall, for instance from \cite[Proposition 14.19]{IK}, that for a modular form $f$ and a Dirichlet character $\chi$, the twist $f \otimes \chi$ is defined by
\begin{equation} \label{eqn:twist-q}
f(z) = \sum_{n} c_f(n) q^n \implies f \otimes \chi (z) = \sum_n c_f(n) \chi(n) q^n
\end{equation}
This is again modular, but generally of level higher than $f$. If $f$ has level $N$ with a character of conductor $N^*$, while $\chi$ is a character modulo $t$, then $f \otimes \chi$ has level equal to the least common multiple $\lcm(N,N^* t,t^2)$ and character equal to that of $f$ multiplied by $\chi^2$. 
In particular, for a quadratic character, $\chi^2=1$ so $f \otimes \chi$ and $f$ have the same character.
We will take $f$ of character $\chi_3$ at level $N=3$, $12$, or $24$ and twist by $\chi_4$. In all these cases, $t=4$ and $\lcm(N,N^*t, t^2) = 48$, so these twists remain in the optimization space.

Another important aspect of the analysis is scaling.
The divisors of $N=48=2^4 \cdot 3$ are 1, 2, 3, 4, 6, 8, 12, 16, 24, 48. There are (non-zero) modular forms for $\chi_3$ at each sub-level 3, 6, 12, 24, 48 and modular forms for $\chi_4$ at 4, 8, 12, 16, 24, 48.
Figure~\ref{fig:48divisors} shows the dimensions of these subspaces.
These dimensions can be computed with Pari by finding a basis for each subspace using \texttt{mfbasis}. 
The spaces for $\chi_3$ and $\chi_4$ are implemented as \texttt{mfinit([48,3,-3])} and \texttt{mfinit([48,3,-4])}, and similarly for divisors of 48, with \texttt{mfinit(\ldots,\{1\})} for the cuspidal space or \texttt{mfinit(\ldots,\{3\})} for the Eisenstein space. 
There are also explicit dimension formulas, which are explained for instance in \cite[section 6.3]{S} and implemented in Pari's command \texttt{mfdim}.

\begin{figure}[t]
\begin{tikzpicture}[xscale=1.618]
\pgfmathsetmacro{\w}{3}
\pgfmathsetmacro{\h}{1}
\pgfmathsetmacro{\x}{0.6}
\draw (0,0) node[draw=black,circle](3){3};
\draw (0,\h) node[draw=black,circle](6){6};
\draw (0,2*\h) node[draw=black,circle](12){12};
\draw (0,3*\h) node[draw=black,circle](24){24};
\draw (\w,0) node[draw=black,circle](4){4};
\draw (\w,\h) node[draw=black,circle](8){8};
\draw (\w,2*\h) node[draw=black,circle](16){16};
\draw (0.5*\w,4*\h) node[draw=black,circle](48){48};
\draw (3)--(6)--(12)--(24)--(48);
\draw (4)--(8)--(16)--(48);
\draw (4)--(12);
\draw (8)--(24);
\draw (3)++(180:\x) node{2};
\draw (6)++(180:\x) node{4};
\draw (12)++(180:\x) node{6+1};
\draw (24)++(180:\x) node{8+4};
\draw (48)++(180:\x) node{12+10};
\draw (48)++(0:\x) node{12+10};
\draw (24)++(0:\x) node{8+4};
\draw (16)++(0:\x) node{6+1};
\draw (12)++(0:\x) node{4+2};
\draw (8)++(0:\x) node{4};
\draw (4)++(0:\x) node{2};
\draw (-\x,-0.5*\h) node{$E_6$, $E_6^*$};
\draw (\w+\x,-0.5*\h) node{$D_6$, $D_6^*$};
\end{tikzpicture}
\caption{Dimensions of the spaces of modular forms for each divisor of $48$. The subspaces for the character $\chi_3$ are written to the left of each circled divisor, and those for $\chi_4$ to the right. The dimensions are written as $e+c$ where $e$ is the dimension of the Eisenstein subspace, and $c$ the dimension of the cuspidal subspace.
At the bottom levels 3 and 4, the theta series of the lattices $E_6$, $E_6^*$, $D_6$, and $D_6^*$ form a basis of Eisenstein series.
A new scaling $z \mapsto sz$, where $s$ divides 12 or 16, applies at each level. These scalings, together with quadratic twists of $E_6$ and $E_6^*$ arising at the top level, form a basis of Eisenstein series.}
\label{fig:48divisors}
\end{figure}
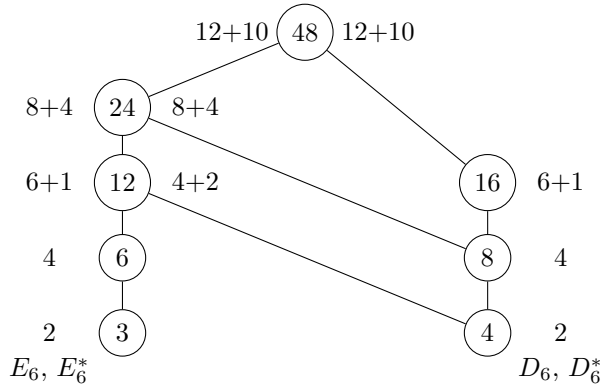

In total, we have a 44-dimensional space of modular forms, with 24-dimensional Eisenstein subspace and 20-dimensional cuspidal subspace. As shown in the chart \ref{fig:48divisors}, there are 2 independent Eisenstein series at each of the bottom levels 3 and 4. These can be thought of as the theta series of $E_6$ and its dual $E_6^*$ for level 3, or $D_6$ and $D_6^*$ for level 4.
Rather than these theta functions, the output of \texttt{mfbasis} involves four Eisenstein series $f_1$, $f_6$, $f_{23}$, and $f_{29}$ which take a simpler form in terms of $\chi_3$ and $\chi_4$:
\begin{align}
f_1(z) &=  \sum_{n=1}^{\infty} q^n \sum_{d \mid n} d^2 \chi_3(n/d) 
= q + 3q^2 + 9q^3 + 13q^4 + \ldots \label{f1} \\
f_6(z) &=\frac{1}{2}L(-2,\chi_3)+ \sum_{n=1}^{\infty} q^n \sum_{d \mid n} d^2 \chi_3(d)
= -\frac{1}{9} + q - 3q^2 + q^3 + 13q^4  +\ldots \label{f6}  \\
f_{23}(z) &= \sum_{n=1}^{\infty} q^n \sum_{d \mid n} d^2 \chi_4(n/d) 
= q + 4q^2 + 8q^3 + 16q^4 + \ldots \label{f23}  \\
f_{29}(z) &=\frac{1}{2}L(-2,\chi_4)+ \sum_{n=1}^{\infty} q^n \sum_{d \mid n} d^2 \chi_4(d) 
= -\frac{1}{4} + q + q^2 - 8q^3 + q^4 +\ldots \label{f29} 
\end{align}
The constant terms involve Dirichlet $L$-values at $1-k$ for weight $k=3$. See \cite[Theorem 4.5.1]{DS1}.

The rest of the Eisenstein basis is given by scaling $f_{1}$ and $f_{6}$ by divisors of $48/3=16$, and scaling $f_{23}$ and $f_{29}$ by divisors of $48/4=12$, as well as two quadratic twists.
Note that twisting by $\chi_4$ increases levels by $4^2$, hence from 3 to 48, as per \cite[Proposition 14.19]{IK}.
The twists are
\begin{align}
f_{11} &= f_1 \otimes \chi_4, \quad f_{11}(z) = \sum_{n=1}^{\infty} q^n \chi_4(n) \sum_{d \mid n} d^2 \chi_3(n/d)  \label{f11}\\
f_{12} &= f_6 \otimes \chi_4, \quad f_{12}(z) = \sum_{n=1}^{\infty} q^n \chi_4(n) \sum_{d \mid n} d^2 \chi_3(d) \label{f12}
\end{align}
The scalings are indexed by the size of $s$ in $z \mapsto sz$, so that
\begin{equation*}
\begin{aligned}
&f_2(z) = f_1(2z),& &f_7(z)=f_6(2z),& &f_{24}(z)=f_{23}(2z),& &f_{30}(z)=f_{29}(2z)& \\
&& && &f_{25}(z)=f_{23}(3z),& &f_{31}(z)=f_{29}(3z)& \\
&f_3(z)=f_1(4z),& &f_8(z) = f_6(4z),& &f_{26}(z) = f_{23}(4z),& &f_{32}(z)=f_{29}(4z)&\\
&& && &f_{27}(z)=f_{23}(6z),& &f_{33}(z)=f_{29}(6z)& \\
&f_4(z)=f_1(8z),& &f_9(z)=f_6(8z),& && && \\
&& && &f_{28}(z)=f_{23}(12z),& &f_{34}(z)=f_{29}(12z)& \\
&f_5(z)=f_1(16z)& &f_{10}(z)=f_6(16z)& && &&
\end{aligned}
\end{equation*}

In the approach of Cohn--Triantafillou, the cuspforms are eventually discarded as a small correction to the Eisenstein series. The method can be applied without knowing them in closed form.
All the necessary information can be obtained automatically using Pari's commands to diagonalize the Hecke operators, calculate Atkin--Lehner transforms, and compute a finite number of Fourier coefficients.

However, our proof of Theorem~\ref{thm:main} uses a candidate with vanishing coefficients $a_n = 0$ for all $n \equiv 1 \bmod 4$. To show that an infinite number of coefficients vanish, we will need to know slightly more about the basis functions. In particular, for the character $\chi_3$, the cuspforms of level 48 can all be expressed as quadratic twists of lower-level basis functions:
\begin{equation} \label{eqn:twists}
\begin{pmatrix} f_{20} \\ f_{21} \\ f_{22} \end{pmatrix} = \begin{pmatrix} 1 & 1 & 0 \\ 3 & 0 & -1 \\ 0 & 2 & -2 \end{pmatrix} \begin{pmatrix} f_{13} \otimes \chi_4 \\ f_{16} \otimes \chi_4 \\ f_{18} \otimes \chi_4 \end{pmatrix}
\end{equation}
The structure of the basis with respect to scalings is as follows:
\begin{align*}
f_{14}(z) &= f_{13}(2z), \quad f_{15}(z) = f_{13}(4z) \\
f_{17}(z) &= f_{16}(2z) \\
f_{19}(z) &= f_{18}(2z) \\
f_{36}(z) &= f_{35}(2z), \quad f_{37}(z) = f_{35}(4z) \\
f_{39}(z) &= f_{38}(2z), \quad f_{40}(z) = f_{38}(4z) \\
f_{42}(z) &= f_{41}(3z)
\end{align*}
The remaining basis elements $f_{13}$, $f_{16}$, $f_{18}$, $f_{35}$, $f_{38}$, $f_{41}$, $f_{43}$, $f_{44}$ are handled using the GP/Pari library \cite{pari}.
Their Fourier expansions begin as follows:
\begin{align*}
f_{13}(z) &= q-3q^3 + 2q^7 + 9q^9 - 22q^{13} + \ldots \\ 
f_{16}(z) &= 2q + 2q^3 - 12q^7 - 14q^9 + 20q^{13} + \ldots \\
f_{18}(z) &= 2q - 14q^3  + 32q^5 - 12q^7 - 46q^9 + 32q^{11} + 20q^{13} + \ldots \\
f_{35}(z) &=2\big(q - q^2 - 2q^4 + 3q^6 + 8q^8 - 3q^9 + 2q^{10} - 6q^{12} + 2q^{13} + \ldots \big) \\
f_{38}(z) &= -2\big(q+2q^2 - 3q^3 - 8q^4 - 2q^5 + 6q^6 + 12q^7 + 8q^8 - 3q^9 - 4q^{10} - \ldots \big) \\ 
f_{41}(z) &= q - 6q^5 + 9q^9 + 10q^{13} + \ldots \\
f_{43}(z) &=2 \big(q + 6q^5 - 3q^9 - 14q^{13} + \ldots \big) \\
f_{44}(z) &= -6 \big(q^3 + 4q^7 - 12q^{11} + \ldots \big)
\end{align*}
Some of them have simple expressions as products of Dedekind eta functions, for example
\[
f_{13}(z)= \eta(2z)^3 \eta(6z)^3 = q \prod_{n=1}^{\infty} (1-q^{2n})^3(1-q^{6n})^3
\]
but we do not make any use of such formulas.

The transforms of the basis functions can be computed using Pari's commands \texttt{mfatkininit} and \texttt{mfatkin}. The former outputs various data, including a matrix representing the action of $z \mapsto -1/(Nz)$ on each basis function, together with a complex floating-point scalar for normalization. We must divide by this factor to compare (\ref{eqn:transform}), which is normalized so that the transform is an involution, with Pari's convention by which the transform has coefficients in the same field as the original modular form (despite the factors $i^k$ and $N^{k/2}$). The second command \texttt{mfatkin} can then be applied to individual forms.

The transform preserves the Eisenstein and cuspidal spaces for each character. It is represented by a block diagonal matrix:
\begin{equation} \label{eqn:w-block-matrix}
W = \begin{pmatrix}
W_{3,e} & & & \\
& W_{3,c} & & \\
&& W_{4,e} & \\
&&& W_{4,c}
\end{pmatrix}
\end{equation}
This is a $44 \times 44$ matrix where the first 12 indices correspond to Eisenstein series for $\chi_3$, the next 10 to cuspforms for $\chi_3$, and then likewise there are 12 and 10 more indices corresponding to Eisenstein series and cuspforms for $\chi_4$.
Missing entries denote $0$ for brevity.
Column by column, $W_{3,e}$ gives the transforms of the Eisenstein series for $\chi_3$ and equals
\[
\setcounter{MaxMatrixCols}{12}
{\tiny
\frac{1}{4}
\begin{pmatrix}
&&&&&&&&& -3/16 &&\\
&&&&&&&&-3/2 & &&\\
&&&&&&&-12 && &&\\
&&&&&& -96 &&& && \\
&&&&&-768 &&&& && \\
&&&&-1/48 &&&&& && \\
&&& -1/6 &&&&&& && \\
&& -4/3 &&&&&&& &&\\
& -32/3 &&&&&&&& && \\
-256/3 &&&&&&&&& && \\
&&&&&&&&&  & & -12 \\
&&&&&&&&&  & -4/3 &
\end{pmatrix}
}
\]
and likewise for cuspforms 
\[
W_{3,c}=\frac{1}{4}
\begin{pmatrix}
& & 1/2 & & & & & & & \\
 & 4 &  & & & & & & &\\ 
32 & & & & & & & & & \\
&&&  & 3/2 &  & 3/2 & & & \\
&&& 12 &  & 12 &  & & & \\
&&&  & -1/6 &  & -3/2 &&& \\
&&& -4/3 &  & -12 &  &&& \\
&&&&&&& 2 & -6 & -16 \\
&&&&&&& 2/3 & 6 & 16/3 \\
&&&&&&& -1 & -3 & -4
\end{pmatrix}
\]
For $\chi_4$, the normalization is different, with $W_{4,e}$ given by 
\setcounter{MaxMatrixCols}{12}
{\tiny
\[
\frac{1}{\sqrt{12}}
\begin{pmatrix}
&&&&&&&&&&&  -1/3 \\
&&&&&&&&&&-8/3 & \\
&&&&&&&&&-9 && \\
&&&&&&&& -64/3 &&& \\
&&&&&&&-72 &&&& \\
&&&&&& -576 &&&&& \\
&&&&& -1/48 &&&&&& \\
&&&& -1/6 &&&&&&& \\
&&& -9/16 &&&&&&&& \\
&&-4/3 &&&&&&&&& \\
& -9/2 &&&&&&&&&& \\
-36 &&&&&&&&&&&
\end{pmatrix}
\]
}
and likewise
\[
W_{4,c} =
\frac{1}{\sqrt{12}} 
\begin{pmatrix}
 &  & 1/4 & & & -1 &&&& \\
 & 2 &  & & -8 & &&&& \\
16 &  & & -64 & & &&&& \\
&& -1/8 & &  & -1/4 &&&& \\
& -1 & & & -2 & &&&& \\
-8 & & & -16 &&  &&&& \\
&&&&&& & 2/3 && \\
&&&&&& 18 & && \\
&&&&&&&& 3 & -3 \\
&&&&&&&& -1 & -3
\end{pmatrix}
\]
The only basis functions with a constant term are $f_6 = -1/9 + \ldots$, $f_{29} = -1/4 + \ldots$, and their scalings $f_{7,8,9,10}$ and $f_{30,31,32,33,34}$. These are proportional to the transforms of $f_{1,2,3,4,5}$ and $f_{23, 24, 25, 26,27,28}$, up to the factors from $W_{3,e}$ and $W_{4,e}$. The coefficient $b_0$ to be maximized therefore takes the form claimed in (\ref{b0}).

\section{Choice of coefficients} \label{sec:choice}

In this section, we specify the values $x_j$ used to define the candidate $g = \sum_j x_j f_j$. Recall that the goal is to maximize a product $b_0 T^3$ where both $g=\sum_n a_n q^n = 1 + a_T q^T + \ldots$ and its transform $\widetilde{g} = \sum_n b_n q^n$ are required to have non-negative coefficients. We choose $T=7$.

The values $x_1, \ldots, x_{44}$ were obtained as follows. First, consider a truncated linear program forcing the constraints from Proposition~\ref{prop:bound} for finitely many values of $n$.
This relies on Pari's command \texttt{mfcoefs} to obtain the coefficients of $f_1, \ldots, f_{44}$ that appear in these constraints, and \texttt{mfatkininit} or \texttt{mfatkin} for the transforms $\widetilde{f_j}$. To solve the linear program, we used the software package GLPK \cite{glpk}. 
The resulting values for $x_j$ are given in floating-point to double precision, and the constraints are only approximately satisfied. We then guess values of $n$ for which a constraint such as $a_n \geq 0$ or $b_n \geq 0$ seems to hold with equality, modulo numerical imprecision. The vanishing coefficients, together with the normalization $a_0=1$, give a system of equations for $x_1, \ldots, x_{44}$ with coefficients in $\Q(\sqrt{N})$, since the Atkin-Lehner transform $\widetilde{g}$ from (\ref{eqn:transform}) involves a factor $N^{k/2}$. Recall that $k=3$ is odd in our case. 

In our case, it appeared that $a_n = 0$ whenever $n \equiv 1 \bmod 4$. As we will see in Section~\ref{sec:1mod4}, this follows from just eleven equations:
\begin{align*}
x_{10} = x_{41} = x_{43} &= 0 \\
x_1 + x_{11} &= 0 \\
 x_6 + x_{12} &= 0 \\
x_{23}+x_{29} &= 0 \\
x_{25} - x_{31} &= 0 \\
x_{35}-x_{38} &= 0 \\
x_{13} + x_{20} + 3x_{21} &= 0 \\
x_{16} + x_{20} + 2x_{22} &= 0 \\
x_{18} - x_{21} -2x_{22} &= 0
\end{align*}
In addition to this infinite progression where $a_n=0$, we impose
\begin{equation} \label{eqn:a0=1}
a_0 = 1
\end{equation}
\begin{align}
a_n = 0 \quad \text{for} \ n \in \{ &2, 3, 4, 6, 8, 10, 11, 12, 22, 26, 32, 38, \label{force}\\
&60, 64, 88, 90, 92, 106, 164, 1932 \} \nonumber
\end{align}
\begin{equation} \label{eqn:transforce}
b_n = 0 \quad \text{for} \ n \in \{1, 2, 3, 4, 7, 8, 9, 10, 13, 14, 36, 82 \}
\end{equation}
Together, these amount to $11+1+20+12=44$ equations for the 44 variables $x_j$.
There is a unique solution, which can be calculated exactly in $\Q(\sqrt{N})$, for instance by working in a polynomial ring modulo $X^2-N$ and using Pari's command \texttt{matinverseimage}.
The fact that $a_{1932}=0$ illustrates that, even for $n \approx 2000$, the Eisenstein part is not yet dominant over the cuspidal part. Moreover, without checking a high enough number of initial coefficients, one would have fewer than 44 equations and the system would not uniquely specify the solution.
In fact, in our first attempt at writing the solution exactly, we had only 43 equations but could determine the one remaining free parameter from the inequalities $a_n, b_n \geq 0$ and the objective to maximize $b_0$.

Table~\ref{table:numbers} shows the solution $x$ truncated to 4 digits.
The exact values are complicated, for instance
\begin{align*}
x_1=&\frac{13305835801389714487894701329762261844642710724107326137255}{5580210204055442306717732912787404128951045878479406537152} \\
&-\frac{5933985634895097866727312605651608472538244579636113805857}{5580210204055442306717732912787404128951045878479406537152}\sqrt{3}
\end{align*}
which is approximately $0.54260880498140096513867653943544603187\ldots$
Approximate values are enough to check any strict inequality $a_n > 0$ or $b_n > 0$. 
The exact quadratic values are important in cases where $a_n$ or $b_n$ is 0, which in the first stage of the calculation could have been a very small negative number. However, these zeros have all been accounted for in setting up the system of equations, so we need only the existence of an exact solution rather than the particular values.

The factors of $1932 = 12 \cdot 7 \cdot 23$ play a role in the vanishing coefficient $a_{1932}=0$, the largest from (\ref{force}). We will see in Section~\ref{sec:eisenstein} that the Eisenstein contribution to $a_n$ is at its smallest relative to $n^2$ when $n=12n_0$, with $n_0 \equiv 5 \bmod 12$ and $n_0$ divisible by many primes $p \equiv 7 \bmod 12$. 

\begin{table}
\caption{Approximate values of the 44 coefficients $x_i \in \Q(\sqrt{3})$ expressing $g=\sum_i x_i f_i$ in the basis from Section~\ref{sec:space}, written as $i=10h + v$ with $h \leq 4$ running horizontally and $v \leq 9$ running vertically. 
The values $x_{10}=x_{41}=x_{43}=0$ are exact. Exact values for the other $x_i$ can be found in a supplementary file \texttt{supplement.txt} available at \url{https://arxiv.org/abs/2211.09044}.
The coefficients $x_1$ through $x_{22}$ correspond to the character $\chi_3$, while $x_{23}$ through $x_{44}$ correspond to $\chi_4$.
The Eisenstein indices are $1 \leq i \leq 12$ and $23 \leq i \leq 34$, with the cuspidal part shaded.
}
\label{table:numbers}
\begin{tabular}{c|rrr|rrr}
& 0 & 1 & 2 &  2 & 3 & 4 \\ \hline
0 &  & $0.0000$ & \cellcolor{gray!40!white}$0.1155$ & & $-0.0880$ & \cellcolor{gray!40!white}$-1.0698$ \\ 
1 & $0.5426$ & $-0.5426$ & \cellcolor{gray!40!white}$0.1669$ & & $-0.3790$  & \cellcolor{gray!40!white}$0.0000$ \\ 
2 & $-1.2415$ & $0.0961$ & \cellcolor{gray!40!white}$-0.0774$ & & $-0.0134$  & \cellcolor{gray!40!white}$-18.0274$ \\ 
3 & $-3.3991$ & \cellcolor{gray!40!white}$-0.6164$ & & $-0.0943$  & $2.8874$ & \cellcolor{gray!40!white}$0.0000$ \\ 
4 & $-1.7265$ & \cellcolor{gray!40!white}$0.0010$ & & $0.0606$  & $-6.7493$ & \cellcolor{gray!40!white}$-1.6689$ \\ 
5 & $-0.8136$ & \cellcolor{gray!40!white}$0.4724$ & & $-0.3790$& \cellcolor{gray!40!white}$-0.4690$ & \\ 
6 & $-0.0961$ & \cellcolor{gray!40!white}$0.0394$ & & $1.4476$ & \cellcolor{gray!40!white}$-1.3134$ & \\ 
7 & $-0.2722$ & \cellcolor{gray!40!white}$-1.0533$ & & $3.5885$ & \cellcolor{gray!40!white}$-1.5459$ & \\ 
8 & $0.5140$ & \cellcolor{gray!40!white}$0.0120$ & &  $-4.8033$ & \cellcolor{gray!40!white}$-0.4690$ & \\ 
9 & $0.4126$ & \cellcolor{gray!40!white}$0.08926$ & & $0.0943$ & \cellcolor{gray!40!white}$-0.8239$ & \\ 
\end{tabular}
\end{table}

\section{The cuspidal part} \label{sec:cusp-total}

In this section, we bound the contributions to $a_n$ and $b_n$ from cuspforms, and prove some observations about the basis functions that will be used to show that $a_n = 0$ for $n \equiv 1 \bmod 4$. 
Both applications rely on a change of basis that is also useful for computational purposes.
\begin{proposition} \label{prop:observations}
The following relations hold between the coefficients of $q^n$ in the basis functions $f_{1},\ldots, f_{44}$. Write $[i]_n$ for the coefficient of $q^n$ in $f_i$.
\begin{enumerate}
\item \label{44n1mod4}
For $n \equiv 1 \bmod 4$, $[44]_n = 0$.
\item \label{41n1mod4}
For $n \not\equiv 1 \bmod 4$, $[41]_n = 0$.
\item \label{3538n1mod4}
For $n \equiv 1 \bmod 4$, $[35]_n = -[38]_n$.
\item \label{22n1mod3}
For $n \equiv 1 \bmod 3$, $[22]_n = 0$.
\item \label{20n2mod3}
For $n \equiv 2 \bmod 3$, $[20]_n = 0$.
\end{enumerate}
\end{proposition}
All of these patterns can be explained by changing to a basis of Hecke eigenforms.
We refer to \cite[Chapter 6]{I} for background on Hecke operators. This is one of the aspects of the framework from \cite{CT} that changes slightly because we consider modular forms with a non-trivial character. In particular, not all the Hecke operators $T_n$ are self-adjoint, and in fact the adjoint of $T_n$ on a space of modular forms with character $\chi$ is $\overline{\chi(n)} T_n$ \cite[Theorem 6.20]{I}. Nevertheless, the $T_n$ are commuting normal operators, and can therefore be diagonalized simultaneously.

\begin{proposition} \label{prop:C}
For all $n \geq 1$, the total contributions from cuspforms to $q^n$ in $g$ or $\widetilde{g}$ are at most:
\begin{align}
|a_{n,{\rm cusp}} | &\leq 21.6161 \cdot n \sigma_0(n) \label{eqn:C-an}\\
|b_{n,{\rm cusp}} | &\leq 24.0266 \cdot  n \sigma_0(n) \label{eqn:C-bn}
\end{align}
\end{proposition}
For comparison, note that the cuspidal part involves a large value $x_{42} \approx -18.0274$ from Table~\ref{table:numbers}.

To apply Theorem~\ref{thm:cusp-bound}, we change to a basis of Hecke eigenforms computed from Pari's commands \texttt{mfheckemat} and \texttt{mateigen}.
The command \texttt{mfheckemat} returns a matrix representing a Hecke operator $T_n$ acting on the cuspidal subspace for $\chi_3$ or $\chi_4$, and \texttt{mateigen} can then give a basis of eigenvectors of this matrix. The output of \texttt{mateigen} in our case involves floating-point approximations to $\sqrt{8}$ and $\sqrt{3}$, which are easy enough to recognize, and can also be justified rigorously by factoring the characteristic polynomial. 

It is enough to consider $f_{13}$, $f_{16}$, $f_{18}$, $f_{20}$, $f_{21}$, $f_{22}$ for $\chi_3$, and $f_{35}$, $f_{38}$, $f_{41}$, $f_{43}$, $f_{44}$ for $\chi_4$. The rest of the basis is obtained by scaling these. 
We divide by the leading coefficient to obtain the following normalized Hecke eigenforms. They come in pairs under complex conjugation, which we denote by $K$. This is an abbreviation for complex conjugation of each coefficient in the $q$-series, not pointwise conjugation as functions of a variable $z$ in the half-plane $\Im(z) > 0$. 
In the same way, we will write $h = \Re(h) + i \Im(h)$ and $Kh = \Re(h) - i\Im(h)$ coefficient-wise.
The basis functions $f_j$ have real coefficients (in fact, integer coefficients).

For $\chi_3$, let
\begin{align*}
h_1 &= f_{13} \\
h_2 &= \frac{1}{4\sqrt{2} i} \left( (-1+\sqrt{8}i)f_{16} + f_{18} \right) \\
h_3 &= \frac{-1}{4\sqrt{2} i} \left( (-1-\sqrt{8}i)f_{16} + f_{18} \right) = Kh_2 \\
h_4 &= \frac{-1}{8} (-f_{20} - 2f_{21} + f_{22} ) \\
h_5 &= \frac{3}{16(1-\sqrt{2} i ) } \left( (2-\sqrt{8}i) f_{20} + \frac{-2+\sqrt{8}i}{3} f_{21} + f_{22} \right) \\
h_6 &= \frac{3}{16(1+\sqrt{2} i ) } \left( (2+\sqrt{8}i) f_{20} + \frac{-2-\sqrt{8}i}{3} f_{21} + f_{22} \right) = Kh_5 \end{align*}
and, for $\chi_4$,
\begin{align*}
h_7 &= \frac{1}{\sqrt{12}i} \left( (1+\sqrt{3}i)f_{35}+f_{38} \right) \\
h_8 &= \frac{-1}{\sqrt{12}i} \left( (1-\sqrt{3}i) f_{35} + f_{38} \right)  = Kh_7\\
h_9 &= \frac{-1}{\sqrt{12}i} \left(-\sqrt{3}i f_{43} + f_{44} \right) \\
h_{10} &= \frac{1}{\sqrt{12}i} \left( \sqrt{3}i f_{43} + f_{44} \right) = Kh_9\\
h_{11} &= f_{41} - 3f_{42} \\
h_{12} &= f_{41} + 3f_{42}
\end{align*}
(in the last two, $f_{42}(z)=f_{41}(3z)$ is included to diagonalize $T_3$; $f_{41}$ is already an eigenform for $T_p$ with $p$ coprime to 48).
The original basis functions are then
\begin{align*}
f_{13} &= h_1 \\
f_{16} &= h_2 + h_3 = 2\Re(h_2)\\
f_{18} &= (1+\sqrt{8}i)h_2 + (1-\sqrt{8}i)h_3 = f_{16} - 2\sqrt{8} \Im(h_2) \\
f_{20} &= \frac{8}{7} ( h_4 + h_5 + h_6) = \frac{8}{7} (h_4 + 2\Re(h_5) ) \\
f_{21} &= \frac{24}{7}h_4  + \big( \frac{-4}{7} - \sqrt{8}i \big)h_5 + \big( -\frac{4}{7} + \sqrt{8}i \big)h_6 = \frac{24}{7}h_4 - \frac{8}{7} \Re(h_5) + 2\sqrt{8}\Im(h_5)\\
f_{22} &= 4\sqrt{2} i (h_6 - h_5 ) = 8\sqrt{2} \Im(h_5)
\end{align*}
for $\chi_3$, and for $\chi_4$,
\begin{align*}
f_{35} &=h_7 + h_8 = 2\Re(h_7) \\
f_{38} &=(-1+\sqrt{3}i)h_7 - (1+\sqrt{3}i)h_8 = -f_{35} - 2\sqrt{3} \Im(h_7)\\
f_{41} &= \frac{1}{2} h_{11} + \frac{1}{2} h_{12} \\
f_{43} &= h_{9} + h_{10} = 2 \Re(h_9) \\
f_{44} &= \sqrt{3}i (h_{10} - h_9) = 2\sqrt{3} \Im(h_9)
\end{align*}

\begin{proof}[Proof of Proposition~\ref{prop:observations}]
The calculation relies on \cite[Theorem 6.20]{I}, which gives the adjoint of a Hecke operator acting on a space of modular forms with character $\chi$:
\[
\langle T_n f, g \rangle = \chi(n) \langle f, T_n g \rangle
\]
In the case that $f=g$ is an eigenform $h$ with eigenvalue $\lambda_n$ for $T_n$, it follows from the conjugate-linearity in the second factor that $\lambda_n = \chi(n) \overline{\lambda_n}$. With $\chi = \chi_4$, this implies that the coefficient of $q^n$ in a Hecke eigenform $h$ is real or imaginary according as $n \equiv \pm 1 \bmod 4$.
Likewise for $\chi_3$, the congruence $n \equiv \pm 1 \bmod 3$ determines whether $h$ has real or imaginary coefficient. 

If $n \equiv 1 \bmod 3$, then we are in the self-adjoint case so $h_5$ has real coefficients. This explains why $f_{22} = 8\sqrt{2} \Im(h_5)$ satisfies $[22]_n = 0$ along this progression, as claimed in item (\ref{22n1mod3}) of Proposition~\ref{prop:observations}. Likewise for item (\ref{20n2mod3}): $f_{20}$ is a real combination of $h_4$ and $2\Re(h_5)$, whereas these would have imaginary coefficients in the skew-adjoint case $n \equiv 2 \bmod 3$, so $[20]_n = 0$.

For $n \equiv 1 \bmod 4$, the eigenforms for $\chi_4$ have real coefficients. It follows that
\[
[38]_n = - [35]_n, \quad [44]_n = 0
\]
as per items (\ref{44n1mod4}) and (\ref{3538n1mod4}). If $n \not\equiv 1 \bmod 4$, then the eigenforms have imaginary coefficients, whereas $f_{41}$ is real, so $[41]_n = 0$ as per item (\ref{41n1mod4}).
\end{proof}

\begin{proof}[Proof of Proposition~\ref{prop:C}]
We apply Deligne's bound Theorem~\ref{thm:cusp-bound} to each of the eigenforms $h_j$. 
For the scalings $s=1,2,3,4$, this gives a multiple of $\frac{n}{s}\sigma_0( \frac{n}{s})$ and we note for simplicity that $\sigma_0(n/s) \leq \sigma_0(n)$.
Some convenient absolute values, in view of the change of basis, are
\[
|1+\sqrt{8}i|=3, \quad \left| \frac{-4}{7} \pm \sqrt{8}i \right| = \frac{12(2+\sqrt{2})}{7}, \quad |\pm 1 + \sqrt{3}i| = 2
\]
The cuspidal coefficients for $\chi_3$ can then be bounded by
\begin{align*}
|[13]_n| &= |[h_1]_n| \leq n \sigma_0(n) \\
|[16]_n| &= |2\Re(h_2)_n| \leq 2 n \sigma_0(n) \\
|[18]_n| &= | (1+\sqrt{8}i)[h_2]_n + (1-\sqrt{8}i)[h_3]_n| \leq 6 n\sigma_0(n) \\
|[20]_n| &= \frac{8}{7}|[h_4]_n+[h_5]_n+[h_6]_n| \leq \frac{24}{7} n\sigma_0(n) \\
|[21]_n| &= \left| \frac{24}{7} [h_4]_n + (-4/7 - \sqrt{8}i)[h_5]_n + (-4/7+\sqrt{8}i)[h_6]_n \right| \leq \frac{24}{7} (3+\sqrt{2}) n \sigma_0(n) \\
|[22]_n| &= 8\sqrt{2}[\Im(h_5)]_n \leq 8\sqrt{2} n \sigma_0(n)
\end{align*}
It follows from (\ref{eqn:twists}) that $f_{20}$, $f_{21}$, $f_{22}$ contribute only for odd $n$; for even $n$ the factor $\chi_4(n)$ vanishes. 
For $\chi_4$, we have
\begin{align*}
[35]_n &= |2\Re(h_7)_n| \leq 2 n\sigma_0(n) \\
[38]_n &= |(-1+\sqrt{3}i)h_7 - (1+\sqrt{3}i)h_8 | \leq 4n\sigma_0(n) \\
[41]_n &\leq n\sigma_0(n) \\
[43]_n &\leq 2n\sigma_0(n)\\
[44]_n &= 2\sqrt{3} |\Im(h_9)_n| \leq 2\sqrt{3} n\sigma_0(n)
\end{align*}

If $n$ is divisible by 12, the cuspidal part is
\[
\begin{aligned}
&x_{13}[13]_n& &+& &x_{14}[13]_{n/2}& &+& &x_{15}[13]_{n/4}& \\
&x_{16}[16]_n& &+& &x_{17}[16]_{n/2}& && && \\
&x_{18}[18]_n& &+& &x_{19}[18]_{n/2}& && && \\
&x_{35}[35]_n& &+& &x_{36}[35]_{n/2}& &+& &x_{37}[35]_{n/4}& \\
&x_{38}[38]_n& &+& &x_{39}[38]_{n/2}& &+& &x_{40}[38]_{n/4}& \\
&x_{41}[41]_n& && && && && &+& &x_{42}[41]_{n/3}&  \\
&x_{43}[43]_n& && && && && && && \\
&x_{44}[44]_n& && && && && && && \\
\end{aligned}
\]
If $n$ is not sufficiently divisible by 2 or 3, we remove the columns for $n/2$, $n/4$, or $n/3$. 
If $\gcd(n,48)=1$ or 3, then we add the contribution from the quadratic twists $f_{20}$, $f_{21}$, and $f_{22}$. 
Case by case, we substitute the values of $x_j$ from Table~\ref{table:numbers} and apply the triangle inequality for an upper bound on $|a_{n,{\rm cusp}}|$.
The upper bound for $|b_{n,{\rm cusp}}|$ has the same form with $x$ replaced by $y=Wx$, where $W$ is the Atkin--Lehner matrix from (\ref{eqn:w-block-matrix}). 
This gives the numbers claimed in Proposition~\ref{prop:C}. 
\end{proof}

The Hecke basis has another important property which we use to compute $a_{n,{\rm cusp}}$. Their coefficients are multiplicative: for a normalized Hecke eigenform $h$,
\begin{equation} \label{eqn:hecke-multiplicative}
h(z) = \sum_n \lambda(n) q^n, \gcd(m,n)=1 \implies \lambda(mn)= \lambda(n) \lambda(n)
\end{equation}
as one can recall for instance from \cite[eq. (6.83), (6.60)]{I}. This makes it possible to compute $\lambda(n)$ in terms of the factors of $n$, which are much smaller than $n$ itself in many of the cases we will need to check.
We will see in Section~\ref{sec:eisenstein} that the Eisenstein part is smallest when $n$ has many factors. 
It is therefore very efficient to compute the cuspidal part via
\[
\lambda( \prod_{p^e \mid \mid n} p^e ) = \prod_{p^e \mid \mid n} \lambda(p^e)
\]
although we will still need to compute cuspidal coefficients at some large primes.

This multiplicative structure gives the following formula for $a_{n,{\rm cusp}}$ in case $n = 2^a 3^b  n_0$ is even. 
It is expressed using an indicator function $\mathbbm{1}_{a \geq 2}$ which is 0 if $n$ is not sufficiently divisible by 2, or 1 if the scalings $f_{15,37,40}$ under $z \mapsto 4z$ are present.
The products are taken over prime powers exactly dividing $n_0$.
Recall that the Hecke basis involves several conjugate pairs such as $h_3 = Kh_2$, which we combine into real and imaginary parts.

\begin{align} \label{an-hecke}
&a_{n,{\rm cusp}} = \\ \nonumber
&\Bigl(x_{13}[h_1]_{2^a} + x_{14} [h_1]_{2^{a-1}} + \mathbbm{1}_{a \geq 2}x_{15}[h_1]_{2^{a-2}}\Bigr)[h_1]_{3^b} \prod_{p^e \mid \mid n_0} [h_1]_{p^e} +  \\ \nonumber
&2\Re\biggl( \Bigl(x_{16}[h_2]_{2^a}+x_{18}(1+\sqrt{8}i)[h_2]_{2^a}+(x_{17}+x_{19}(1+\sqrt{8}i)[h_2]_{2^{a-1}}) \Bigr)[h_2]_{3^b} \prod [h_2]_{p^e}\biggr) \\ \nonumber
&+2\Re \Biggl( \biggl( x_{35}[h_7]_{2^a}+x_{36}[h_7]_{2^{a-1}}+x_{37}\mathbbm{1}_{a\geq 2}[h_7]_{2^{a-2}} +  \\ \nonumber
& \hspace{1.5cm} (-1+\sqrt{3}i)\Bigl(x_{38}[h_7]_{2^a}+x_{39}[h_7]_{2^{a-1}} + x_{40}\mathbbm{1}_{a\geq 2}[h_7]_{2^{a-2}} \Bigr) \biggr)[h_7]_{3^b} \prod [h_7]_{p^e} \Biggr) \nonumber 
\end{align}

Even values of $n$ are the most important for us. When we generate the list of values for which $a_{n,{\rm cusp}}$ needs to be computed, the last odd value is $n=315$. We already know $a_n \geq 0$ up to that point since these inequalities are among the constraints of the linear program that was solved in the first place. For $n$ of this size, it is also easy to compute $a_{n,{\rm cusp}}$ in the original basis. The Hecke basis becomes decisive for larger values of $n$, where an attempt to compute $a_{n,{\rm cusp}}$ directly in the original basis exhausted our available memory.
Using (\ref{an-hecke}) instead, the calculation is only a matter of time.


If $n$ were odd, (\ref{an-hecke}) would need to be modified by adding the quadratic twists $f_{20}, f_{21}, f_{22}$ from (\ref{eqn:twists}), as well as $f_{41}$, $f_{42}$, $f_{43}$, $f_{44}$. The latter do not contribute for even $n$ because they are combinations of Hecke eigenforms for $\chi_4$, as in Proposition~\ref{prop:observations} parts (\ref{44n1mod4}) and (\ref{41n1mod4}). 

Grouping terms as in (\ref{an-hecke}) could also be used to obtain somewhat better values of $C$ and $C^*$ in Proposition~\ref{prop:C}, but ultimately this is not much help with such a small $\varepsilon > 0$ from the Eisenstein part in Proposition~\ref{prop:epsilon}.

\section{Vanishing on a progression} \label{sec:1mod4}

In this section, we prove that the finite number of linear relations we have assumed between the variables $x_1, \ldots, x_{44}$ are enough to guarantee infinitely many vanishing coefficients.
As stated in Proposition~\ref{prop:mf}, we show
\begin{equation} \label{eqn:1mod4zeros}
n \equiv 1 \bmod 4 \implies a_n = 0
\end{equation}
To prove this, we analyze the contributions from Eisenstein series and cuspforms for both of the characters $\chi_3$ and $\chi_4$. Each of the four vanishes separately.

The Eisenstein contribution vanishes because of four linear relations between the variables:
\begin{equation} \label{eqn:eisenstein-vanishing}
x_1 + x_{11} = x_6 + x_{12}=x_{23}+x_{29} = x_{25} - x_{31} = 0 .
\end{equation}
These lead to pairwise cancellations between the Eisenstein series $f_1$, $f_{11}$, $f_{6}$, $f_{12}$, $f_{23}$, $f_{29}$, $f_{25}$, and $f_{31}$. 
Recall that the coefficient of $q^n$ in each of these is given by
\[
\begin{aligned}
&f_1: \sum_{d \mid n} d^2 \chi_3(n/d),& && &f_{11}: \chi_4(n) \sum_{d \mid n} d^2 \chi_3(n/d)& \\
&f_6: \sum_{d \mid n} d^2 \chi_3(d),& && &f_{12}: \chi_4(n) \sum_{d \mid n} d^2 \chi_3(d)& \\
&f_{23}: \sum_{d \mid n} d^2 \chi_4(n/d),& && &f_{29}: \sum_{d \mid n} d^2 \chi_4(d)& \\
&f_{25}(z) = f_{23}(3z),& && &f_{31}(z) = f_{29}(3z)&
\end{aligned}
\]
If $n \equiv 1 \bmod 4$, then $\chi_4(n) = 1$, while (\ref{eqn:eisenstein-vanishing}) gives $x_1 = -x_{11}$ and $x_6 = -x_{12}$. Thus the contributions to $q^n$ from $f_1$ and $f_{11}$ cancel out. Likewise for $f_6$ and $f_{12}$.
The Eisenstein series for $\chi_3$ thus cancel in pairs.

To see the cancellation between $f_{23}$ and $f_{29}$, we observe that each term vanishes in the sum over divisors $d \mid n$. Indeed, since $x_{29} = -x_{23}$ from (\ref{eqn:eisenstein-vanishing}),
\begin{align*}
x_{23} \sum_{d \mid n} d^2 \chi_4(n/d) + x_{29} \sum_{d \mid n} d^2 \chi_4(d) &= x_{23} \left( \sum_{d \mid n} d^2 \big( \chi_4(n/d) - \chi_4(d) \big) \right) \\
&= x_{23} \sum_{d \mid n} d^2 \cdot 0
\end{align*}
Here, $\chi_4(n/d) = \chi_4(d)$ because $n \equiv 1 \bmod 4$. 

Something similar happens for $f_{25}$ and $f_{31}$. These two are rescalings of $f_{23}$ and $f_{29}$, which only contribute if $n$ is divisible by 3. In that case, again assuming $n \equiv 1 \bmod 4$, we have $n = 3 \cdot n/3$ where $n/3 \equiv -1 \bmod 4$. 
From (\ref{eqn:eisenstein-vanishing}), we have $x_{25}=x_{31}$. The contribution of this pair is then
\begin{align*}
x_{25} \sum_{d \mid n/3} d^2 \chi_4 \left( \frac{n/3}{d} \right) + x_{31} \sum_{d \mid n/3} d^2 \chi_4(d) &= x_{25} \sum_{d \mid n/3} d^2 \left( \chi_4 \left( \frac{n/3}{d} \right) + \chi_4(d) \right) \\
&= x_{25} \sum_{d \mid n/3} d^2 \cdot 0
\end{align*}
This time, $\chi_4(n/3d)$ and $\chi_4(d)$ have the opposite sign because $n/3 \equiv -1 \bmod 4$. 

The cuspforms for $\chi_4$ contribute nothing because of the relations
\begin{equation} \label{eqn:cusp4relations}
x_{35} = x_{38}, \quad x_{41}=x_{43}=0.
\end{equation}
There is no contribution from $f_{41}$ and $f_{43}$ because $x_{41}=x_{43}=0$, and the rest can be analyzed using
Proposition~\ref{prop:observations}. For $n \equiv 1 \bmod 4$, there is no contribution from $f_{44}$ and the coefficients of $q^n$ in $f_{35}$ and $f_{38}$ are negatives of each other. The latter cancel because $x_{35}=x_{38}$. 
If $n$ is divisible by 3, then there is a further term $x_{42} [41]_{n/3}$ from the scaling $f_{42}(z) = f_{41}(3z)$. In this case, $n = 3 \cdot n/3$ where $n/3 \equiv 3 \bmod 4$, so $[41]_{n/3}=0$ by Proposition~\ref{prop:observations} again.

Finally, the cuspforms for $\chi_3$ contribute 
\[
x_{13}[13]_n + x_{16}[16]_n + x_{18}[18]_n + x_{20}[20]_n + x_{21}[21]_n + x_{22}[22]_n
\]
Recall from (\ref{eqn:twists}) that $f_{20}$, $f_{21}$, $f_{22}$ are linear combinations of $f_{13} \otimes \chi_4$, $f_{16}\otimes \chi_4$, $f_{18} \otimes \chi_4$. The sum becomes
\begin{align*}
&x_{13}[13]_n + x_{16}[16]_n + x_{18}[18]_n + \\
&\chi_4(n)\left( x_{20}([13]_n + [16]_n) +  x_{21}(3[13]_n - [18]_n) + x_{22} (2[16]_n - 2[18]_n) \right)
\end{align*}
For $n \equiv 1 \bmod 4$, $\chi_4(n) = 1$ and we can collect terms as follows
\begin{align*}
[13]_n (x_{13}+x_{20} + 3x_{21} ) +
[16]_n (x_{16}+x_{20} + 2x_{22} )+ 
[18]_n (x_{18}-x_{21} - 2x_{22} )
\end{align*}
This is $0+0+0$ because of the linear relations
\begin{equation} \label{eqn:cusp3relations}
x_{13}+x_{20}+3x_{21} = x_{16}+x_{20}+2x_{22} = x_{18}-x_{21}-2x_{22} = 0.
\end{equation}

\section{The Eisenstein part} \label{sec:eisenstein}

In this section, we show that the Eisenstein series contribute at least a positive multiple of $n^2$ to $a_n$ and $b_n$, except for the progression we have already noted where $a_n = 0$.

\begin{proposition} \label{prop:epsilon}
For $n \equiv 1 \bmod 4$, $a_{n,{\rm eis}} = 0$. 
For $n \not\equiv 1 \bmod 4$,
\[
a_{n,{\rm eis}} > \varepsilon n^2, \qquad \varepsilon = 8.753 \times 10^{-6}
\]
For all $n \geq 1$,
\[
b_{n,{\rm eis}} > \varepsilon^* n^2, \qquad \varepsilon^* = 1.358 \times 10^{-3}
\]
\end{proposition}

The Eisenstein part is dictated by the factors of $n = 2^a \cdot 3^b \cdot n_0$, leading to several cases over which we take the minimum to obtain the numerical values in Proposition~\ref{prop:epsilon}. Table~\ref{table:eps} gives admissible values for $\varepsilon$ and $\varepsilon^*$ depending on the exponents $a$ and $b$.

\begin{table}
\begin{tabular}{cr|ll}
$b$ & $a$ & $\varepsilon$ & $\varepsilon^*$  \\ \hline
$0$ & $0$ & $0.6468$ &   $0.1612$ \\ 
 & $1$ & $0.003742$ & \cellcolor{pink}$0.001358$ \\  
 & $2$ & $0.0008134$ & $0.03902$ \\ 
 & $3$ & $0.0008264$ & $0.2038$ \\ 
 & $\geq 4$ & $0.0002649$ & $0.01181$ \\ 
$\geq 1$ & $0$ & $0.7414$ & $0.2757$ \\ 
 & $1$ & $0.006758$ & $0.05462$ \\  
 & $2$ & \cellcolor{pink}$0.000008753$ & $ 0.01363$ \\ 
 & $3$ &      $0.01524$ &$0.1377$ \\ 
 & $\geq 4$ & $0.01635$& $0.1008$\\ 
\end{tabular}
\caption{Admissible values in the inequalities $a_{n,{\rm eis}} \geq \varepsilon n^2$ and $b_{n,{\rm eis}} \geq \varepsilon^* n^2$ for $n=2^a 3^b n_0$ where $n_0$ is not divisible by 2 or 3. The case $a_{n,{\rm eis}} = 0$ for $n \equiv 1 \bmod 4$ is treated separately. The worst cases, shown in red, are $a=2, b\geq 1$ for $a_n$ and $a=1,b=0$ for $b_n$.
}
\label{table:eps}
\end{table}

The Eisenstein contribution is expressed in terms of four basic functions:
\begin{align}
\sigma_3^+(n) = \sum_{d \mid n} d^2 \chi_3(n/d), \quad \sigma_3^{-}(n)= \sum_{d \mid n} d^2 \chi_3(d), \label{sigma3} \\ \sigma_4^+(n) = \sum_{d \mid n} d^2 \chi_4(n/d), \quad \sigma_4^-(n)= \sum_{d \mid n} d^2 \chi_4(d) \label{sigma4}
\end{align}
In this notation, the basis functions from (\ref{f1})--(\ref{f29}) are $f_1(z) = \sum_n \sigma_3^+(n) q^n$, $f_{6}(z) =-1/9+ \sum_n \sigma_3^{-}(n) q^n$, $f_{23}(z) = \sum_n \sigma_4^+(n) q^n$, and $f_{29}(z) = -1/4+\sum_n \sigma_4^{-}(n) q^n$. The remaining Eisenstein series are scalings or quadratic twists of these four.

From their definition as sums over divisors, the four functions $\sigma = \sigma_{3,4}^{\pm}$ are multiplicative:
\begin{equation} \label{eqn:multi}
\gcd(m,n)=1 \implies \sigma(mn) = \sigma(m)\sigma(n)
\end{equation}
Their values on prime powers can be evaluated as geometric series. Assuming $\chi(p) = \pm 1$, where $p$ is prime,
\begin{align}
&\sigma^-(p^e) = \frac{(p^2 \chi(p) )^{e+1} - 1}{p^2 \chi(p) - 1} \label{geo} \\ 
&\sigma^+(p^e) = \frac{p^{2e+2}-\chi(p)^{e-1}}{p^2-\chi(p)} \label{geo-plus}
\end{align}
\begin{equation} \label{eqn:chip}
\sigma^-(p^e) = \chi(p)^e \sigma^{+}(p^e)
\end{equation}
and likewise
\begin{equation} \label{eqn:chin}
\sigma^-(n) = \chi(n) \sigma^+(n)
\end{equation}
as long as $n$ has no factors for which $\chi$ vanishes.
However, there are special cases when $\chi(p)=0$ and the sums reduce to a single term: $\sigma^+(p^e)=p^{2e}$, $\sigma^-(p^e)=1$. In particular,
\begin{equation} \label{eqn:2power3power}
\sigma_3^+(3^e) = 3^{2e}, \sigma_4^+(2^e) = 2^{2e}, \qquad \sigma_3^-(3^e)=1 = \sigma_4^-(2^e)
\end{equation}
The other values follow the pattern of (\ref{geo-plus}) with $\chi_3(2)=\chi_4(3)=-1$:
\begin{equation} \label{eqn:23cross-powers}
\sigma_3^+(2^e) = \frac{4^{e+1}+(-1)^e}{5}, \qquad \sigma_4^+(3^e) = \frac{9^{e+1}+(-1)^e}{10}
\end{equation}
For composite $n$, the values $\sigma(n)$ are then given by products. It follows that the Eisenstein coefficients are of order $n^2$, except that $\sigma^{-}(n)$ can be smaller when $n$ is highly divisible by 2 or 3. 

To analyze this further, write $n= 2^a 3^b n_0$ where $\gcd(n_0,48)=1$. In any of the four cases $\sigma = \sigma_{3,4}^{\pm}$,
\begin{equation} \label{eqn:abn0}
\sigma(n) = \sigma(2^a) \sigma(3^b) \sigma(n_0).
\end{equation}
The factors $\sigma(2^a)$ and $\sigma(3^b)$ can be evaluated with (\ref{eqn:chip}), (\ref{eqn:2power3power}), and (\ref{eqn:23cross-powers}).
The Eisenstein contribution $a_{n,{\rm eis}}$ or $b_{n,{\rm eis}}$ is a sum of up to 22 such terms.
We must consider the factorization $n = 2^a 3^b n_0$ to see which terms arise, but in any case (\ref{eqn:chin}) allows us to write the sum as
\begin{equation} \label{eqn:regroup}
a_{n,{\rm eis}} = \sigma_3^+(n_0) X_3 + \sigma_4^+(n_0) X_4, \quad b_{n,{\rm eis}} = \sigma_3^+(n_0)Y_3 + \sigma_4^+(n_0)Y_4
\end{equation}
where $X_3$ and $X_4$ are linear combinations of the variables $x_j$ (or $y_j$ for the transform, with the corresponding linear combinations $Y_3$ and $Y_4$).
We factor out $\sigma_3^+(n_0)$:
\begin{equation} \label{eqn:sigma3out}
a_{n,{\rm eis}} = \sigma_3^+(n_0) \left( X_3 + \frac{\sigma_4^+(n_0)}{\sigma_3^+(n_0)} X_4 \right)
\end{equation}
Depending on the sign of $X_4$ (or $Y_4$ for the transform), we use either an upper or a lower bound for the ratio of $\sigma_3^+(n_0)$ and $\sigma_4^+(n_0)$. 
Bounding the ratio is more efficient than considering the worst case for each separately, and for some $n$, it is important to be careful. See (\ref{eqn:sums-ratio}) below for the most delicate case.
\begin{proposition} \label{prop:ratio-4-3}
For $n_0$ not divisible by $2$ or $3$,
\begin{equation} \label{eqn:products}
\prod_{p \equiv 7 \bmod 12} \frac{p^2-1}{p^2+1} \leq \frac{\sigma_4^+(n_0)}{\sigma_3^+(n_0)} \leq \prod_{p \equiv 5 \bmod 12} \frac{p^2+1}{p^2-1}
\end{equation}
where the products are over all primes congruent to $5$ or $7 \bmod 12$. 
Numerically,
\begin{equation} \label{eqn:prodeuler}
0.94999 \leq \frac{\sigma_4^+(n_0)}{\sigma_3^+(n_0)} \leq 1.09696
\end{equation}
Separately, still for $n_0$ not divisible by $2$ or $3$, the two functions satisfy
\begin{align}
0.9429 < \prod_{\substack{p \equiv 2 \bmod 3 \\ p \neq 2}} (1-p^{-2}) \leq \frac{\sigma_3^+(n_0)}{n_0^2} \leq \prod_{p \equiv 1 \bmod 3} (1+p^{-2}) < 1.0336 \label{eqn:sigma3-n0} \\
0.9631 < \prod_{\substack{p \equiv 3 \bmod 4 \\ p \neq 3}} (1-p^{-2}) \leq \frac{\sigma_4^+(n_0)}{n_0^2} \leq \prod_{p \equiv 1 \bmod 4} (1+p^{-2}) < 1.0545 \label{eqn:sigma4-n0}
\end{align}
\end{proposition}

\begin{proof}
Let $p^e$ be the prime powers dividing $n_0$ exactly. By the geometric series evaluating $\sigma(p^e)$, the ratio in (\ref{eqn:products}) is
\[
\frac{\sigma_4^+(n_0)}{\sigma_3^+(n_0)} = \prod_{p^e \mid \mid n_0} \left( \frac{p^{2e+2}-\chi_4(p)^{e-1}}{p^2-\chi_4(p)} \right)\left( \frac{p^2 - \chi_3(p)}{p^{2e+2}-\chi_3(p)^{e-1}} \right)
\]
If $\chi_4(p) = \chi_3(p)$, then the corresponding factor is 1. Only the primes congruent to 5 or 7 modulo 12 contribute, so that $\chi_4(p)$ and $\chi_3(p)$ have opposite sign.
If $e$ is odd, then $\chi_4(p)^{e-1}=\chi_3(p)^{e-1}$ in either case.
It follows that
\begin{align*}
\frac{\sigma_4^+(n_0)}{\sigma_3^+(n_0)} = \prod_{ \substack{p \equiv 5, \\ e \ {\rm odd}} } \frac{p^2+1}{p^2-1} \prod_{ \substack{p \equiv 7, \\ e \ {\rm odd}} } \frac{p^2-1}{p^2+1} \prod_{ \substack{p \equiv 5, \\ e \ {\rm even}} } \frac{p^2+1}{p^2-1} \frac{p^{2e+2}-1}{p^{2e+2}+1}  \prod_{ \substack{p \equiv 7, \\ e \ {\rm even}} }  \frac{p^2-1}{p^2+1} \frac{p^{2e+2}+1}{p^{2e+2}-1}
\end{align*}
For $p \equiv 5$, we have $(p^2+1)/(p^2-1) > 1$, whereas $(p^2-1)/(p^2+1) < 1$ for $p \equiv 7 \bmod 12$. These inequalities are reversed for the terms $(p^{2e+2} \pm 1)/(p^{2e+2} \mp 1)$, so if present, those extra factors bring the product farther from the maximum or minimum. The extreme cases are where $n_0$ is a product of many primes all congruent to 5, for the upper bound, or to 7 for the lower bound, with multiplicities $e=1$ throughout.

A similar argument gives (\ref{eqn:sigma3-n0}) and (\ref{eqn:sigma4-n0}) in parallel. We have
\[
\frac{\sigma^+(n_0)}{n_0^2} = \prod_{p^e \mid \mid n_0} \frac{p^2 \chi(p) - \chi(p)^e p^{-2e} }{p^2 \chi(p) - 1}
\]
The factor for primes with $\chi(p)=1$ satisfies
\[
1 \leq \frac{p^2 - p^{-2e}}{p^2-1} \leq \frac{p^2 - p^{-2}}{p^2-1} = 1+ p^{-2}
\]
and for $\chi(p)=-1$,
\[
1 \geq \frac{p^2 + (-1)^e p^{-2e}}{p^2+1} \geq 1 - p^{-2}
\]
The upper and lower bounds (\ref{eqn:sigma3-n0}) and (\ref{eqn:sigma4-n0}) follow by extending the products to all $p$, regardless of whether they divide $n_0$. 
\end{proof}

We can now prove Proposition~\ref{prop:epsilon}. We describe the calculations in detail only for some cases, and simply indicate which terms to include for the others. 

\subsection{Values $n$ coprime to 48: $a=0$, $b=0$, $n=n_0$}
To begin, let us estimate the Eisenstein contribution to $q^n$ assuming that $n$ has no factors in common with the level $N=48$. Then we need only consider $f_1$, $f_6$, $f_{11}$, $f_{12}$, $f_{23}$, and $f_{29}$, without any further scalings. Their contribution to $a_n$ is 
\begin{align*}
\big(x_1 + x_{11} \chi_4(n) \big)\sigma_3^+(n) + \big(x_6 + x_{12} \chi_4(n) \big) \sigma_3^-(n) + x_{23}\sigma_4^+(n) + x_{29}\sigma_4^{-}(n)
\end{align*}
and similarly for $b_n$, if one changes the vector of coefficients $x$ to $y=Wx$.
Collecting terms gives
\[
\big( x_1+x_{11}\chi_4(n) + x_6\chi_3(n) +x_{12}\chi_3(n)\chi_4(n) \big)\sigma_3^+(n) + \big( x_{23}+x_{29}\chi_4(n) \big) \sigma_4^+(n) 
\]
which is of the form (\ref{eqn:regroup}) with
\begin{align*}
X_3 &=  x_1+x_{11}\chi_4(n) + x_6\chi_3(n) +x_{12}\chi_3(n)\chi_4(n) \\
X_4 &= x_{23} + x_{29}\chi_4(n)
\end{align*}
If $\chi_4(n)=1$, then $X_3 = X_4 = 0$ because $x_1+x_{11}=x_6+x_{12}=x_{23}+x_{29}=0$, so we see once again that $a_{n,{\rm eis}}=0$ for $n \equiv 1 \bmod 4$. 
If $\chi_4(n)=-1$, then $X_4 = x_{23}-x_{29} \approx -0.18866 < 0$, so we use the upper bound from (\ref{eqn:prodeuler}) to see how much could be subtracted. 
It follows that $a_{n,{\rm eis}} > 0.6468 n^2$ if $\chi_3(n)=1$, or $1.0093n^2$ if $\chi_3(n)=-1$. 

Similar estimates apply to $b_n$, but without the cancellation in case $\chi_4(n)=1$. In the four cases according to $\chi_3(n) = \pm 1$ and $\chi_4(n) = \pm 1$,
\begin{align*}
y_1 + y_{11} + y_6 + y_{12} &\approx -0.1032 < 0, \quad y_{23}+y_{29} \approx 0.6783 \\
y_1 + y_{11} - y_6 - y_{12} &\approx -0.4734 < 0\\
y_1 - y_{11} - y_6 + y_{12} &\approx 0.4649 > 0, \quad y_{23}-y_{29}\approx0.6205 \\
y_1 - y_{11} + y_6 - y_{12} &\approx 0.11168 > 0
\end{align*}
This time, $Y_4 = y_{23} \pm y_{29} > 0$, so we use the lower bound from (\ref{eqn:prodeuler}). 
It follows that $b_{n,{\rm eis}}/n^2$ is at least 0.5103 if $n \equiv  1 \bmod 12$, 0.1612 if $n \equiv 5$,  0.9942 if $n \equiv 11 \bmod 12$, or 0.6611 if $n \equiv 7 \bmod 12$ (listed from top to bottom in the cases above).

\subsection{Odd multiples of 3: $\gcd(n,48)=3$, $a=0$, $b \geq 1$, $n_0 = n/3^b$}

This case offers all the same terms as for $\gcd(n,48)=1$, plus two more corresponding to $f_{25}(z) = f_{23}(3z)$ and $f_{31}(z)=f_{29}(3z)$.
The two new terms, with extra factors $\sigma_4^{\pm}(3^b)$ and $\sigma_4^{\pm}(3^{b-1})$, lead to
\begin{align*}
X_3 &=  9^bx_1+9^b\chi_4(n)x_{11} + x_6\chi_3(n_0) +x_{12}\chi_3(n_0)\chi_4(n) \\
X_4 &= \frac{9^{b+1}+(-1)^b}{10}x_{23} + (-1)^b\frac{9^{b+1}+(-1)^b}{10} x_{29}\chi_4(n_0) + \\
&\frac{9^b+(-1)^{b-1}}{10}x_{25} +(-1)^{b-1} \frac{9^b+(-1)^{b-1}}{10}x_{31}\chi_4(n_0)
\end{align*}
The sign $\chi_4(n) = (-1)^b \chi_4(n_0)$ can occur as $\chi_4(n)$ directly from the twists $f_{11}$ and $f_{12}$, or as $\chi_4(n_0)$ when we collect terms using (\ref{eqn:chin}) to convert $\sigma_4^-$ to $\sigma_4^+$. 
If $\chi_4(n)=1$, then $(-1)^{b-1}\chi_4(n_0) = -1$, so
\begin{align*}
X_3 &= 9^b (x_1 + x_{11}) + \chi_3(n_0) (x_6 + x_{12}) \\
X_4 &= \frac{1}{10} \left( (9^{b+1}+(-1)^b)(x_{23}+x_{29}) + (9^b + (-1)^{b-1}) (x_{25}-x_{31})\right)
\end{align*}
The relations $x_1+x_{11} = x_{6}+x_{12}=x_{23}+x_{29}=x_{25}-x_{31}$ entail $X_3 = X_4=0$, giving another explanation of the vanishing of $a_{n,{\rm eis}}$ when $n \equiv 1 \bmod 4$. 

If $\chi_4(n)=-1$, then instead
\begin{align*}
X_3 &= 9^b (x_1 - x_{11}) + \chi_3(n_0) (x_6 - x_{12}) \\
X_4 &= \frac{1}{10} \left( (9^{b+1}+(-1)^b)(x_{23}-x_{29}) + (9^b + (-1)^{b-1}) (x_{25}+x_{31})\right)
\end{align*}
with $x_1 - x_{11} > 0$. 
The expected size is $n^2 = 9^b n_0^2$, where $b \geq 1$, so we write
\begin{align*}
X_3 &\geq 9^b ( (x_1 - x_{11}) - |x_6-x_{12}|/9 ) \\
X_4 &\leq - \frac{1}{10}9^b ( | 9x_{23}-9x_{29}+x_{25}+x_{31}| + |x_{29}+x_{25}+x_{31}|/9 )
\end{align*}
We take the upper bound from (\ref{eqn:prodeuler}) in case $X_4$ is negative, check that the total is positive, and multiply with the lower bound from (\ref{eqn:sigma3-n0}).
The result is that $a_{n,{\rm eis}} > 0.7414 \cdot 9^b n_0^2 = 0.7414 n^2$. 

The Eisenstein contribution to $b_n$ is similar, with $y = Wx$ in place of $x$, but $y_1 + y_{11} < 0$ (in fact, $y_1 = 0$, but that plays no role in this case). The positivity comes from $Y_4$ instead of $Y_3$. If $\chi_4(n)=1$, then
\begin{align*}
Y_3 &= 9^b (y_1 + y_{11}) + \chi_3(n_0) (y_6 + y_{12}) \leq -9^b(|y_1 + y_{11}|+|y_6+y_{12}|/9) \\
Y_4 &= \frac{1}{10} \left( (9^{b+1}+(-1)^b)(y_{23}+y_{29}) + (9^b + (-1)^{b-1}) (y_{25}-y_{31})\right) \\
&\geq \frac{1}{10}9^b \left(9y_{23} + 9y_{29} + y_{25}-y_{31}- \frac{|y_{23}+y_{29}-y_{25}+y_{31}|}{9}  \right)
\end{align*}
These estimates for $Y_3$ and $Y_4$ show, first of all, that $Y_4 > 0$, and then, using the lower bound from (\ref{eqn:prodeuler}), that $Y_3 + \sigma_4^+(n_0)/\sigma_3(n_0) Y_4 > 0.2924$. 
It follows that $b_{n,{\rm eis}} > 0.2757 n^2$ in this case, using the lower bound from (\ref{eqn:sigma3-n0}).
If $\chi_4(n)=-1$, then positive contributions from both $Y_3$ and $Y_4$ lead to an even better bound.

\subsection{Multiples of 2}

When $n$ is even, $\chi_4(n)=0$ so the twists $f_{11}$ and $f_{12}$ disappear. 
The Eisenstein contribution to $a_n$ is
\begin{equation} \label{eqn:all-scales}
\begin{aligned}
a_{n,{\rm eis}} = \\
&x_1 \sigma_3^+(n)& &+& &x_6 \sigma_3^{-}(n)& &+& &x_{23}\sigma_4^+(n)& &+& &x_{29}\sigma_4^{-}(n)& \\
&x_2 \sigma_3^+(n/2)& &+& &x_7 \sigma_3^{-}(n/2)& &+& &x_{24}\sigma_4^+(n/2)& &+& &x_{30}\sigma_4^{-}(n/2)& \\
&x_3 \sigma_3^+(n/4)& &+& &x_8 \sigma_3^{-}(n/4)& &+& &x_{26}\sigma_4^+(n/4)& &+& &x_{32}\sigma_4^{-}(n/4)& \\
&x_4 \sigma_3^+(n/8)& &+& &x_9 \sigma_3^{-}(n/8)& && && && && && \\
&x_5 \sigma_3^+(n/16)& &+& &x_{10}\sigma_3^{-}(n/16)& && && && && && \\
&& && && &+& &x_{25}\sigma_4^+(n/3)& &+& &x_{31}\sigma_4^{-}(n/3)& \\
&& && && &+& &x_{27} \sigma_4^+(n/6)& &+& &x_{33}\sigma_4^{-}(n/6)& \\
&& && && &+& &x_{28}\sigma_4^{+}(n/12)& &+& &x_{34}\sigma_4^{-}(n/12)&
\end{aligned}
\end{equation}
with possible contributions from the scales $n$, $n/2$, $n/4$, $n/8$, $n/16$, $n/3$, $n/6$, and $n/12$ if $n$ is sufficiently divisible by 2 and 3. 
The same formula can be used equally well if $n$ does not have enough factors of 2 or 3, provided we change the vector $x$ by resetting the corresponding entries $x_j$ to 0.
Likewise, the Eisenstein contribution to $b_n$ has the same form with $x$ replaced by $y=Wx$ where $W$ is the matrix from (\ref{eqn:w-block-matrix}), and perhaps with some $y_j$ changed to 0 depending on the factors of $n$.

Write $n = 2^a 3^b n_0$ where $\gcd(n_0, 48)=1$. 
Each scale $n/s$ has the same form up to changing the values $a,b$. 
Collecting terms gives (\ref{eqn:regroup}) with
\begin{equation} \label{eqn:pm3}
X_3 = 9^b X_3^+ + (-1)^a \chi_3(n_0) X_3^{-}
\end{equation}
where
\begin{align*}
X_3^+ &=x_1 \sigma_3^+(2^a) + x_2\sigma_3^+(2^{a-1}) + x_3 \sigma_3^+(2^{a-2}) + x_4\sigma_3^+(2^{a-3}) + x_5\sigma_3^+(2^{a-4}) \\
X_3^- &=x_6 \sigma_3^+(2^a) - x_7\sigma_3^+(2^{a-1}) + x_8\sigma_3^+(2^{a-2}) - x_9 \sigma_3^+(2^{a-3}) + x_{10} \sigma_3^+(2^{a-4})
\end{align*}
and
\begin{equation} \label{eqn:pm4}
X_4 = 4^a X_4^{+} + (-1)^b \chi_4(n_0) X_4^{-}
\end{equation}
where
\begin{align*}
X_4^{+} &= \sigma_4^+(3^b) \left( x_{23}+\frac{x_{24}}{4} + \frac{x_{26}}{4^{2}} \right) + \sigma_4^+(3^{b-1}) \left( x_{25} + \frac{x_{27}}{4} + \frac{x_{28}}{4^2} \right) \\
X_4^{-} &= \sigma_4^+(3^{b})(x_{29}+x_{30}+x_{32}) - \sigma_4^+(3^{b-1})(x_{31}+x_{33}+x_{34})
\end{align*}
It can be convenient to factor out the expected size $n^2 = 4^a \cdot 9^b \cdot n_0^2$.
Substituting the geometric series from (\ref{eqn:2power3power}) gives
\begin{align} \label{eqn:all-scales-collect}
a_{n,{\rm eis}} = \frac{\sigma_3^+(n_0)4^a 9^b}{5} &\left( 4x_1 + x_2 + \frac{x_3}{4} + \frac{x_4}{16} + \frac{x_5}{64} + \right.  \\ 
&\left. \frac{(-1)^a}{4^a} \big(x_1 - x_2 + x_3 - x_4 + x_5 \big) \right. \nonumber \\
& \left. + \frac{(-1)^a \chi_3(n_0)}{9^b} \left( 4x_6 - x_7 + \frac{x_8}{4} - \frac{x_9}{16} + \frac{x_{10}}{64} + \right. \right. \nonumber \\
 &\left. \left. \frac{(-1)^a}{4^a} \big(x_6 + x_7 + x_8 + x_9 + x_{10} \big) \right) \right)  \nonumber \\
+\frac{\sigma_4^+(n_0)4^a 9^b}{10} &\left( 9\big( x_{23} + \frac{x_{24}}{4} + \frac{x_{26}}{16} \big) + x_{25}+\frac{x_{27}}{4}+\frac{x_{28}}{4^2} \right. \nonumber \\
&\left.  \frac{(-1)^b}{9^b} \big( x_{23}+\frac{x_{24}}{4}+\frac{x_{26}}{16} - x_{25} - \frac{x_{27}}{4} - \frac{x_{28}}{16} \big) \right. \nonumber \\
 &\left. + \frac{(-1)^b\chi_4(n_0)}{4^a} \left( 9(x_{29} + x_{30} + x_{32} ) - x_{31} - x_{33} - x_{34} + \right. \right. \nonumber \\
&\left. \left.  \frac{(-1)^b}{9^b} (x_{29}+x_{30}+x_{32}+x_{31}+x_{33}+x_{34}) \right) \right) \nonumber
\end{align}
assuming $n$ is divisible by 48 so that all terms are present. 
We describe below which terms to remove, depending on the factors of $n$.
The worst cases are when $\gcd(n,48)$ is 12 or 2, for $a_{n,{\rm eis}}$ and $b_{n,{\rm eis}}$ respectively.

\subsection{The case $\gcd(n,48)=48$}

In this case, $a \geq 4$ and $b \geq 1$, so the terms with factors $4^{-a}$ or $9^{-b}$ are relatively small.
Write the subsums appearing in (\ref{eqn:all-scales-collect}) as
\begin{align*}
S_1&=4x_1+x_2+x_3/4+x_4/4^2+x_5/4^3\\
S_2&=9(x_{23}+x_{24}/4+x_{26}/4^2)+x_{25}+x_{27}/4+x_{28}/4^2\\
S_3&=4x_6-x_7+x_8/4-x_9/4^2+x_{10}/4^3\\
S_4&=x_{23}+x_{24}/4+x_{26}/4^2-x_{25}-x_{27}/4-x_{28}/4^2\\
S_5&=x_1-x_2+x_3-x_4+x_5\\
S_6&=9(x_{29}+x_{30}+x_{32})-x_{31}-x_{33}-x_{34}\\
S_7&=x_6+x_7+x_8+x_9+x_{10}\\
S_8&=x_{29}+x_{30}+x_{32}+x_{31}+x_{33}+x_{34}
\end{align*}
In this notation,
\begin{align*}
X_3 &= \frac{1}{5} \left( 4^a9^b S_1 + 4^a (-1)^a \chi_3(n_0)S_3 +9^b (-1)^a S_5 + \chi_3(n_0) S_7  \right) \\
X_4 &= \frac{1}{10} \left( 4^a 9^b S_2 + 4^a (-1)^b S_4 + 9^b (-1)^b\chi_4(n_0)S_6 + \chi_4(n_0)S_8 \right) \\
\end{align*}
For $a_{n,{\rm eis}}$, there is a positive contribution from $S_2$.
Since $a \geq 4$ and $b \geq 1$,
\[
X_4 \geq \frac{1}{10}4^a 9^b \left( S_2 - \frac{|S_4|}{9} - \frac{|S_6|}{4^4} - \frac{|S_8|}{9 \times 4^4} \right) > 0.02785 \cdot 4^a 9^b
\]
while
\[
X_3 \leq -\frac{1}{5}4^a 9^b \left(|S_1|+\frac{|S_3|}{9}+\frac{|S_5|}{4^4}+\frac{|S_7|}{9 \times 4^4} \right) < -0.0091177 \cdot 4^a 9^b
\]
Since $X_4 > 0$, we use the lower bound (\ref{eqn:prodeuler}) and conclude that
\[
a_{n,{\rm eis}} \geq \sigma_3^+(n_0) \left(X_3 + \frac{\sigma_4^+(n_0)}{\sigma_3^+(n_0)} X_4 \right) > 0.01635 \cdot n_0^2 4^a 9^b = 0.01635 n^2
\]

We have $b_{n,{\rm eis}} > 0.1008 n^2$ by a similar use of the triangle inequality, with positive contributions from both $Y_3$ and $Y_4$ leading to a better bound than for $a_{n,{\rm eis}}$.

\subsection{$a \geq 4$, $b = 0$: the case $\gcd(n,48)=16$}

We discard the terms involving $x_{25, 27, 28,31,33,34}$ from (\ref{eqn:all-scales-collect}) since $n$ is not divisible by 3. The sums $S_3$ and $S_4$, which previously were negligible because of the factor $9^{-b}$, should now be regarded as part of the main term. We consider the cases $a=4$ and $a \geq 5$ separately, because a naive application of the triangle inequality using $4^{-a} \leq 4^{-4}$ would lead to a negative lower bound.
In this way, we find that $a_{n,{\rm eis}} \geq 0.0003723 n^2$ if $a=4$, or $0.0002649n^2$ if $a \geq 5$. Table~\ref{table:eps} records only the minimum of these two for simplicity, so that the same bound holds for all $a \geq 4$. 

\subsection{$a=3$, $b \geq 1$: the case $\gcd(n,48)=24$}

We discard the terms involving $x_{5}$ and $x_{10}$, and we can substitute $(-1)^a = -1$. 
Terms that were neglected because of a factor $4^{-a}$ can now be treated exactly.

\subsection{$a=3$, $b=0$: the case $\gcd(n,48)=8$}

We discard the terms corresponding to multiples of 16, namely $x_{5,10}$, as in the case $a = 3$, $b \geq 1$, as well as $x_{25,27,28,31,33,34}$ which correspond to multiples of 3.  

\subsection{$a=2$, $b \geq 1$: the case $\gcd(n,48)=12$}

We discard $x_{4,5,9,10}$ and set $(-1)^a=1$. Let us detail only the case $b=1$, which is the worst one for $a_{n,{\rm eis}}$. For $b \geq 2$, the triangle inequality gives a better bound $\varepsilon = 1.0003 \times 10^{-3}$ using $9^{-b} \leq 9^{-2}$ to neglect terms. For $b_n$, we obtain $\varepsilon^* = 0.01363$ in case $b=1$, and a better bound $\varepsilon^* = 0.02609$ in case $b \geq 2$. It is important to distinguish those cases since a reckless use of the triangle inequality would give a negative bound for $b=1$.

Taking then the case $a=2,b=1$ for $a_{n,{\rm eis}}$, we have
\[
X_3 = S_1 + \chi_3(n_0)S_3, \qquad X_4 = S_2 + \chi_4(n_0) S_4
\]
where
\begin{align*}
S_1 &= 117x_1 + 27x_2 + 9x_3 \\
S_3 &= 13x_6 - 3x_7 + x_8 \\
S_2 &= 8(16x_{23} + 4x_{24}+x_{26})+16x_{25}+4x_{27}+x_{28} \\
S_4 &= -8(x_{29}+x_{30}+x_{32})+x_{31}+x_{33}+x_{34}
\end{align*}
Substituting the values for $x_j$, we find that $S_1$ and $S_4$ are negative, while $S_2$ and $S_3$ are positive. The least favourable signs for proving a lower bound are therefore $\chi_3(n_0)=-1$ and $\chi_4(n_0)=1$, in which case we need $X_4 = S_2 + S_4$ to beat $X_3=S_1 - S_3$. Numerically, $S_1 - S_3 \approx -0.7110$, $S_2+S_4 \approx 0.7498$, and their ratio
\begin{equation} \label{eqn:sums-ratio}
\frac{S_3-S_1}{S_2+ S_4} \approx 0.9482
\end{equation}
is only slightly less than 0.9499 from (\ref{eqn:prodeuler}). We would not be able to prove $a_{n,{\rm eis}} > 0$ in this case using only the worse ratio 0.931869 obtained from separate upper and lower bounds on $\sigma_4^+(n_0)$ and $\sigma_3^+(n_0)$. Proposition~\ref{prop:ratio-4-3} gives 
\[
a_{n,{\rm eis}} = \sigma_3^+(n_0)\left( X_3 + \frac{\sigma_4^+(n_0)}{\sigma_3^+(n_0)} X_4 \right) > 8.7536 \times 10^{-6} \ n^2
\]
remembering to divide by 144 from $n^2 = 144n_0^2$. 

Some of the coefficients really are this small: it is not only a failure of our bounds. 
For example, with $n=12n_0$ where $n_0 = 7 \cdot 59$, $7 \cdot 71$, $7 \cdot 83$, and $7 \cdot 107$, we have
\[
 \frac{a_{n,{\rm eis}} }{n^2} \times 10^5 \approx 6.2643, \qquad 6.2649, \qquad 6.2652, \qquad 6.2656
\]
Primes congruent to 7 mod 12 act to minimize (\ref{eqn:prodeuler}), while primes congruent to 11 are neutral and can be included to arrange $n_0 \equiv 5$ so that the signs align with the worst case for $X_3$ and $X_4$.
Taking a smoother $n$ of this form decreases the coefficient further:
\[
n = 544236 = 12 \cdot 7 \cdot 11 \cdot 19 \cdot 31 \implies \frac{a_{n,{\rm eis}}}{n^2} \approx 2.382567 \times 10^{-5}
\]

\subsection{$a=2$, $b=0$: the case $\gcd(n,48)=4$}

We have the same terms as for $n=12n_0$, but without $x_{25,27,28,31,33,34}$ which correspond to multiples of 3.

\subsection{$a=1$, $b \geq 1$: the case $\gcd(n,48)=6$}

We discard $x_{3,4,5,8,9,10,26,28,32,34}$ because $n$ is not divisible by 4.

\subsection{$a=1$, $b=0$: the case $\gcd(n,48)=2$}

This is the worst case for $b_{n,{\rm eis}}$. 
\[
\begin{aligned}
&X_3 = S_1 + \chi_3(n_0)S_3,&  &X_4 = S_2 + \chi_4(n_0)S_4& \\
&S_1 = 3x_1 + x_2& &S_2 = 4x_{23}+x_{24}&\\
&S_3 = -3x_6 + x_7& &S_4 = x_{29}+x_{30} & 
\end{aligned}
\]

For the values $y_j$, the signs are $S_1 < 0$, $S_3 > 0$, $S_2 > 0$, $S_4<0$, so the combinations $S_1 - S_3$ and $S_2 + S_4$ lead to the smallest value when $n_0 \equiv 5 \bmod 12$:
\[
b_{n,{\rm eis}} > 0.005432 \cdot n_0^2 = 0.001358 \cdot n^2
\]
since $n^2 = 4 n_0^2$ introduces a final factor of 4. 
For comparison, some relatively small values $b_{n,{\rm eis}} \div n^2 \approx 0.001791777$ and 0.0016229 occur if $n = 2 \cdot 7 \cdot 19 \cdot 53$ or $n=2 \cdot 7 \cdot 11 \cdot 19 \cdot 31$, which are products of primes congruent to 7 mod 12 with one extra factor 11 or 53 fixing the congruence class $n_0 \equiv 5 \bmod 12$. 

\section{A relation between $D_6$, $D_6^*$, and $\Z^6$} {\label{sec:d6-z6}

The choice of optimization space, using the character $\chi_4$, was motivated in part by the fact that it contains the theta series of some familiar lattices: $D_6$, $D_6^*$, and $\Z^6$. 
Here, $D_6$ is the lattice of integer vectors with even coordinate sum, and $D_6^*$ is its dual (see \cite[p. 117]{splag}).
In the course of proving Theorem~\ref{thm:main}, we noticed a special feature of dimension 6, which is the following linear relation between these theta series:
\begin{equation} \label{eqn:theta-identity}
\vartheta_{D_6}(z)+4\vartheta_{D_6^*}(4z)=5\vartheta_{\Z^6}(2z)
\end{equation}
A linear relation of this form between $D_n$, $D_n^*$, and $\Z^n$ seems specific to $n=6$ (with $n=1,2$ as somehow degenerate cases where $D_n$ is congruent to $\Z^n$). Matching the first two coefficients would suggest an identity between, on the one hand, $\vartheta_{D_n}(z)+(n-2)\vartheta_{D_n^*}(sz)$ for some scaling $s$, and $(n-1)\vartheta_{\Z^n}(2z)$ on the other. For most values of $n$ and $s$, the higher coefficients do not match.

One strategy, knowing that these theta series are modular with the same properties, is to simply check enough coefficients in the $q$-series of both sides.
The identity (\ref{eqn:theta-identity}) is also easy to see from the expressions of the three theta functions in terms of the basis $f_1,\ldots, f_{44}$ used above. It amounts to
\[
(64f_{23} - 4f_{29}) + 4(4f_{23} - 4f_{29}) = 5(16f_{23}-4f_{29})
\]

It is also possible to give an algebraic proof.
The theta series of $D_n$, $D_n^*$, and $\Z^n$ are given in terms of the Jacobi theta functions. The latter are
\begin{equation*}
\theta_2(z)=\sum_{m=-\infty}^{\infty} e^{\pi i (m+1/2)^2 z}, \quad \theta_3(z)=\sum_{m=-\infty}^{\infty} e^{\pi i m^2 z}, \quad \theta_4(z) = \sum_{m=-\infty}^{\infty} (-1)^m e^{\pi i m^2 z}
\end{equation*}
In terms of these,
\begin{equation} \label{eqn:thetapowers}
\vartheta_{\Z^n}=\theta_3^n, \quad \vartheta_{D_n}=\frac{\theta_3^n+\theta_4^n}{2}, \quad \vartheta_{D_n^*} = \theta_2^n + \theta_3^n
\end{equation}
We will use the following identities, taken from \cite[p.104]{splag}
\begin{align}
\theta_2^4+\theta_4^4 &= \theta_3^4 \label{eqn:theta4} \\
\theta_3(z)+\theta_4(z) &= 2\theta_3(4z) \label{eqn:theta34} \\
\theta_3(z)-\theta_4(z) &= 2\theta_2(4z) \label{eqn:theta24} \\
\theta_3(z)^2 - \theta_4(z)^2 &= 2\theta_2(2z)^2 \label{eqn:theta22} \\
\theta_3(z)^2+\theta_4(z)^2 &=2\theta_3(2z)^2 \label{eqn:theta32} \\
\theta_3(z)\theta_4(z) &= \theta_4(2z)^2 \label{eqn:theta42}
\end{align}
The first of these is Jacobi's identity on sums of four squares. The others are related to the special cases $n=1,2$ where $D_n$ is congruent to $\Z^n$. 

We now take $n=6$. To express $\vartheta_{\Z^6}(2z)$, take $z \mapsto 2z$ in (\ref{eqn:theta4}) and multiply by $\theta_3(2z)^2$:
\[
\vartheta_{\Z^6}(2z) = \theta_3(2z)^6 = \theta_2(2z)^4\theta_3(2z)^2 + \theta_4(2z)^4 \theta_3(2z)^2
\]
To see $D_6$ and $D_6^*$ on the right, we substitute (\ref{eqn:theta22}), (\ref{eqn:theta32}), (\ref{eqn:theta42}) for each term $\theta_j(2z)$. This counteracts the scaling $z \mapsto 2z$ from the beginning: 
\[
\vartheta_{\Z^6}(2z) = \left( \frac{\theta_3(z)^2-\theta_4(z)^2}{2}\right)^2\left(\frac{\theta_3(z)^2+\theta_4(z)^2}{2}\right) + \theta_3(z)^2 \theta_4(z)^2 \left(\frac{\theta_3(z)^2+\theta_4(z)^2}{2}\right)
\]
Multiplying these terms leads to
\begin{equation} \label{eqn:express-z6}
\vartheta_{\Z^6}(2z) = \frac{\theta_3(z)^6+\theta_4(z)^6+3\theta_3(z)^4\theta_4(z)^2+3\theta_3(z)^2\theta_4(z)^4}{8}
\end{equation}
The 5 in our target identity comes from a binomial expansion, using (\ref{eqn:theta24}) and (\ref{eqn:theta34}):
\begin{align*}
\vartheta_{D_6^*}(4z)&=\theta_2(4z)^6 + \theta_3(4z)^6 = \left( \frac{\theta_3(z)-\theta_4(z)}{2} \right)^6 + \left( \frac{\theta_3(z)+\theta_4(z)}{2}\right)^6 \\
&= \frac{\theta_3(z)^6+\theta_4(z)^6 + 15\theta_3(z)^4\theta_4(z)^2 + 15\theta_3(z)^2\theta_4(z)^4}{2^5}
\end{align*}
We solve this for the cross-terms:
\[
3\theta_3(z)^4\theta_4(z)^2+3\theta_3(z)^2\theta_4(z)^4 = \frac{2^5 \vartheta_{D_6^*}(4z) - \theta_3(z)^6-\theta_4(z)^6}{5}
\]
Substituting this in (\ref{eqn:express-z6}) gives
\begin{align*}
\vartheta_{\Z^6}(2z) &= \frac{1}{8} \left( \theta_3(z)^6 + \theta_4(z)^6 + \frac{2^5\vartheta_{D_6^*}(4z)-\theta_3(z)^6 - \theta_4(z)^6}{5}\right) \\
&= \frac{1}{5}\left( \frac{\theta_3(z)^6 + \theta_4(z)^6}{2} + \frac{2^5}{8} \vartheta_{D_6^*}(4z) \right) \\
&= \frac{1}{5} \left( \vartheta_{D_6}(z) + 4\vartheta_{D_6^*}(4z) \right)
\end{align*}
as required.

\section*{Acknowledgements}

It is a pleasure to thank both Henry Cohn and Henri Cohen for their helpful advice. Thanks also to Rupert Li, Henry Cohn, David de Laat, and Andrew Salmon for sharing results from earlier versions of their articles \cite{L} and \cite{CdLS}. Thanks again to David de Laat for advice on computing the Cohn--Elkies function with \cite{dLL}.
Thanks to Claire Burrin for suggesting reference \cite{BK}.
Thanks to the referees for their comments and efforts reading the paper.

\section*{Funding}

Viazovska's research was supported by Swiss National Science Foundation project 184927.
Dostert was partially supported by the Wallenberg AI, Autonomous Systems and Software Program (WASP) funded by the Knut and Alice Wallenberg Foundation.

\end{document}